\documentclass[reqno,11pt]{amsart}

\usepackage{graphicx, amsmath, amssymb, amscd, mathtools, hyperref, amsthm, euscript, psfrag, amsfonts,bm,color}

\DeclareMathAlphabet{\mathpzc}{OT1}{pzc}{m}{it}

\newtheorem{thm}{Theorem}[section]
\newtheorem{lem}[thm]{Lemma}
\newtheorem{prop}[thm]{Proposition}
\newtheorem{cor}[thm]{Corollary}
\newtheorem*{cor*}{Corollary}
\newtheorem*{lem*}{Lemma}

\numberwithin{thm}{section}

\theoremstyle{definition}
\newtheorem{defn}{Definition}[section]
\numberwithin{equation}{section}

\newtheorem{Rmk}[equation]{Remark}

\theoremstyle{plain}

\def\CC{{\mathbb C}}

\def\NN{{\mathbb N}}

\def\RR{{\mathbb R}}

\def\ZZ{{\mathbb Z}}

\def\vecu{{\text{\boldmath$u$}}}
\def\vecv{{\text{\boldmath$v$}}}

\def\vecw{{\text{\boldmath$w$}}}
\def\vecx{{\text{\boldmath$x$}}}

\def\vectheta{{\text{\boldmath$\theta$}}}

\def\vecchi{{\text{\boldmath$\chi$}}}

\def\scrA{{\mathcal A}}
\def\scrB{{\mathcal B}}
\def\scrC{{\mathcal C}}

\def\scrH{{\mathcal H}}
\def\scrG{{\mathcal G}}
\def\scrI{{\mathcal I}}

\def\scrM{{\mathcal M}}

\def\scrS{{\mathcal S}}

\def\scrU{{\mathcal U}}

\def\scrW{{\mathcal W}}

\def\dim{\operatorname{dim}}
\def\dist{\operatorname{dist}}

\def\SL{\operatorname{SL}}

\def\SO{\operatorname{SO}}

\def\T{\operatorname{T{}}}

\def\Onder#1#2#3#4#5{#1 \setbox0=\hbox{$#1$}\setbox1=\hbox{$#2$}
       \dimen0=.5\wd0 \dimen1=\dimen0 \dimen2=\dp0 \dimen3=\dimen2
       \advance\dimen0 by .5\wd1 \advance\dimen0 by -#4
       \advance\dimen1 by -.5\wd1 \advance\dimen1 by -#4
       \advance\dimen2 by -#3 \advance\dimen2 by \ht1
       \advance\dimen2 by 0.3ex \advance\dimen3 by #5
        \kern-\dimen0\raisebox{-\dimen2}[0ex][\dimen3]{\box1}
       \kern\dimen1}

\newcommand{\GaG}{\Gamma\backslash G}

\newcommand{\sfrac}[2]{{\textstyle \frac {#1}{#2}}}

\newcommand{\fg}{\mathfrak{g}}
\newcommand{\fk}{\mathfrak{k}}

\newcommand{\fa}{\mathfrak{a}}

\newcommand{\fn}{\mathfrak{n}}

\newcommand{\Ad}{\mathrm{Ad}}

\newcommand{\mBR}{{m^{\mathrm{BR}}}}
\newcommand{\mBMS}{{m^{\mathrm{BMS}}}}
\newcommand{\mBRast}{{m^{\mathrm{BR}_*}}}

\newcommand{\Uus}{\mathcal{U}(\upsilon,s)}

\newcommand{\ba}{\backslash}
\newcommand{\bH}{\mathbb H}
\newcommand{\op}{\operatorname}
\newcommand{\Mm}{\scrM}
\newcommand{\la}{\lambda}
\renewcommand{\epsilon}{\varepsilon}
\newcommand{\BR}{\op{BR}}
\newcommand{\BMS}{\op{BMS}}
\newcommand{\br}{\mathbb R}\newcommand{\bms}{m^{\BMS}}
\newcommand{\m}{m}
\begin{document}
\title[Exponential mixing]{Spectral gap and Exponential mixing on geometrically finite hyperbolic manifolds}

\author{Sam Edwards}
\thanks{S.\ E.\ was supported by postdoctoral scholarship 2017.0391 from the Knut and Alice Wallenberg Foundation and
H.\ O.\ was supported by NSF grants}
\address{Department of Mathematics, Yale University, New Haven, CT 06520 }
\email{samuel.edwards@yale.edu}

\author{Hee Oh}
\address{Department of Mathematics, Yale University, New Haven, CT 06520 and Korea Institute for Advanced Study, Seoul, Korea }

\email{hee.oh@yale.edu}



\subjclass[2000]{}


\begin{abstract} Let $\scrM=\Gamma\ba \bH^{d+1}$ be a geometrically finite hyperbolic manifold with critical exponent exceeding $d/2$. 
We obtain a precise asymptotic expansion of the matrix coefficients for the geodesic flow in $L^2(\T^1(\scrM))$, with exponential error term essentially as good as the one given by the spectral gap for the Laplace operator on $L^2(\scrM)$
 due to Lax and Phillips. Combined with the work of Bourgain, Gamburd, and Sarnak and its generalization by Golsefidy and Varju on expanders, this implies \emph{uniform} exponential mixing for congruence covers of $\scrM$ when $\Gamma$ is a Zariski dense subgroup contained in an arithmetic subgroup
 of $\op{SO}^\circ(d,1)$. \end{abstract}

\maketitle
\tableofcontents
\section{Introduction}
Let $\scrM=\Gamma\ba \bH^{d+1}$ be a geometrically finite hyperbolic manifold, where $\Gamma$ is a discrete subgroup of $G=\SO^\circ(d+1,1)$. We suppose that the critical exponent $\delta$ of $\Gamma$ is strictly bigger than $d/2$.

Denoting by $\Delta$ the negative of  the Laplace operator on $\scrM$, the bottom eigenvalue of $\Delta$ on $L^2(\Mm)$ is known to be simple and given as $\la_0=\delta(d-\delta)$ by Patterson \cite{Pa} and Sullivan \cite{Su}. Lax and Phillips  \cite{LaxPhillips} showed that there are only finitely many eigenvalues of $\Delta$ on $L^2(\Mm)$ in the interval $[\la_0,\frac{d^2}4)$. We denote by $\la_1$ the smallest eigenvalue of $\Delta$ in $(\la_0, \frac{d^2}4)$, and define $s_1=\frac{d}{2}+\sqrt{(\frac{d}{2})^2-\lambda_1}$. Then $s_1$ is the unique number in $(\frac{d}{2},\delta)$ such that
$$0<\delta(d-\delta)=\lambda_0< \lambda_1=s_1 (d-s_1)< d^2/4. $$
If there is no eigenvalue in the open interval $(\la_0, \frac{d^2}4)$,
we set $\la_1=d^2/4$, i.e., $s_1=d/2$.

Denote by $\mathcal G^t$ the geodesic flow on the unit tangent bundle $\T^1(\Mm)$ and by $dx$ the Liouville measure on $\T^1(\Mm)$;
this is a $\mathcal G^t$-invariant Borel measure which is infinite when $\delta<d$.

The main result of this paper is that the Lax-Philips spectral gap, that is,  $\la_1-\la_0$, or equivalently $\delta-s_1$, controls {\it rather precisely} the exponential convergence rate of the correlation function for the geodesic flow $\mathcal G^t$ acting on $L^2(\T^1(\Mm), dx)$:

\begin{thm} \label{expmixing}\label{main} Let $\delta>d/2$, and set $\eta:=\min \{ \delta-s_1, 1\}$. There exists $m>d(d+1)/2$ such that
for any $\epsilon>0$ and functions $\psi_1, \psi_2$ on $\T^1(\Mm)$ with  $\scrS^{m} (\psi_1),\scrS^{m} (\psi_2)<\infty$, we have, as $t\to +\infty$,
\begin{align*}
e^{(d-\delta)t}\int_{\T^1(\scrM)}\psi_1\big(\scrG^t(x)\big)\psi_2(x)\,dx =& \frac{1}{m^{\BMS}(\T^1(\Mm))} m^{\BR}(\psi_1)m^{\BR_*}(\psi_2)
\\&+ O_{\epsilon}\left( e^{(-\eta+\epsilon) t}\scrS^m(\psi_1)\scrS^m(\psi_2)\right),
\end{align*}
where  
\begin{itemize}
\item $m^{\BMS}$ is the Bowen-Margulis-Sullivan measure on $\T^1(\Mm)$;
\item $m^{\BR}$ and $m^{\BR_*}$ are, respectively, the unstable and stable
Burger-Roblin measures on $\T^1(\Mm)$, which are defined compatibly with the choice of $dx$ and $dm^{\BMS}$;
\item $\scrS^m(\psi_i)$ denotes the $L^2$-Sobolev norm of $\psi_i$ of degree $m$ and  the implied constant depends only on $\epsilon$.
\end{itemize}
\end{thm}

We note that $|m^{\BR}(\psi_i)|<\infty$ when $\scrS^{m} (\psi_i)<\infty$ (see Lemma \ref{mMeasSobbd}), and hence the main term above is well-defined. For $\psi_i$ compactly supported, the asymptotic formula (without error term) holds for any $\delta>0$, as was obtained by Roblin \cite{Roblin}.

\begin{Rmk}\label{mainthmrmk}
\begin{enumerate}
\item When $\Mm$ has finite volume, i.e., when $\delta=d$, exponential mixing of the geodesic flow  is a classical result due to  Ratner
for $d=1$ and to Moore \cite{Moore} for general $d\ge 1$. Combining Moore's proof with Shalom's trick \cite[proof of Theorem 2.1]{Shalom} yields a rate of mixing $\eta=\min\lbrace d-s_1,1\rbrace$ as in Theorem \ref{expmixing} above. This proof  makes explicit use of the fact \cite{Hirai}
that there are no non-spherical complementary series representations $\scrU(\upsilon,s)$ with $s>d-1$ (see below for notation). Our proof does not require this statement. However, combining the aforementioned fact with our proof improves the rate of mixing for the geodesic flow on finite-volume hyperbolic manifolds to $\eta=\min\lbrace d-s_1, 2\rbrace$. In addition, if there is no non-spherical complementary series
representation  $\scrU(\upsilon,s)$ with $s>s_1+1$ appearing in $L^2(\Gamma\ba G)$, then $\eta$ can be taken
to be $d-s_1$ (see Remark \ref{final} for details).

\item For $d=1$ and $2$, Theorem \ref{expmixing} can be deduced from \cite{BKS} and \cite{Vi}, respectively. For a general $d\ge 1$, the case of $\delta>\max\{d-1, d/2\}$ was obtained in \cite{MoOh}  for {\it{some}} $\eta>0$ which is not explicit. 

\item The main novelties of Theorem \ref{expmixing} are that
it addresses \emph{all} geometrically finite groups with $\delta>d/2$ (even those with cusps), and that it gives an optimal value of $\eta$ with respect to the dependency on the spectral gap $\delta-s_1$. 

\item The order $m$ of Sobolev norm required is in principle completely computable; it may be chosen independently of $\Gamma$, and satisfies $m=O(d^2)$. 

\end{enumerate}
\end{Rmk}

The BMS measure $m^{\BMS}$ is known to be the unique measure of maximal entropy (which is $\delta$)
for the geodesic flow (\cite{Su}, \cite{Otal}). Babillot showed that the geodesic flow is mixing with respect to $m^{\BMS}$ \cite{Ba}.
Theorem \ref{expmixing} is known to imply
 the following exponential mixing for the BMS measure (see \cite[Theorem 1.6]{MoOh} for compactly supported functions, and  \cite[Theorem 1.9]{KelmerOh} for general bounded functions):

\begin{thm}\label{expbmsmixing}
There exist  $\beta>0$ (explicitly computable, depending only on $\eta$ in Theorem \ref{expmixing}) and $m> d(d+1)/2$ such that
for all bounded functions $\psi_1, \psi_2 $ on $\T^1(\scrM)$ supported on the one-neighbourhood of $\mathrm{supp}(\mBMS)$, we have, as $t\to \infty$,
\begin{align*}
\int_{\T^1(\scrM)}\psi_1\big(\scrG^t(x)\big)\psi_2(x)\,d\mBMS(x)&= \frac{1}{m^{\BMS}(\T^1(\Mm))}  \mBMS(\psi_1)\mBMS(\psi_2)
\\&+ O\left( e^{-\beta t}\|\psi_1\|_{C^m}  \| \psi_2\|_{C^m}\right)
\end{align*}
where $\|\psi_i\|_{C^m}$ denotes the $C^m$-norm of $\psi_i$.
\end{thm}

\begin{Rmk} 
When $\Gamma$ is convex cocompact, Theorem \ref{expbmsmixing}, for some $\beta>0$ (which is not explicit), follows from the work of Stoyanov \cite{Stoyanov}, which is based on symbolic dynamics and Dolgopyat operators (see also \cite{SW}). This result in turn implies Theorem \ref{expmixing} with implicit $\eta>0$  (see \cite{OhWinter}, \cite{SW}).
\end{Rmk}

Let $\Gamma$ be a Zariski dense subgroup of an arithmetic subgroup $G(\mathbb Z)$ of $G$.
Denote by $\Gamma_q$ the congruence subgroup of $\Gamma$ of level $q$: $\Gamma_q=\{\gamma\in \Gamma: \gamma=e\text{ mod $q$}\}$. When $\pi^{-1}(\Gamma)$ satisfies the strong approximation property for the spin covering map $\pi:\op{Spin}(d+1,1)\to G$, the work of Bourgain-Gamburd-Sarnak \cite{BGS2} and its generalization by Golsefidy-Varju \cite{SV} on expanders imply
that
there exists a finite set $S$ of primes such that the family $\mathcal F:=\{\Gamma_q: \text{$q$ is square-free with no factors in $S$}\}$ has a uniform spectral gap in the sense
that 
\begin{equation} \label{sp} \inf_{\Gamma_q\in \mathcal F} \{ \delta -s_1(q)\} >0,
\end{equation}
where $s_1(q) (d-s_1(q))$ is the second smallest eigenvalue of the Laplacian $\Delta$ on $L^2(\Gamma_q\ba \bH^{d+1})$ \footnote{In view of the recent preprint \cite{HS} and an upcoming work by He and de Saxc\'e, the square-free condition in
 $\mathcal F$ may be removed.}.
This uses the transfer property from the combinatorial spectral gap to the archimedean spectral gap due to \cite{BGS} (see also \cite{Kim}). For certain families of subgroups the uniform spectral gap \eqref{sp} may be explicitly bounded from below, see  \cite{BC,BurgerSarnak,Gamburd,Magee,Sarnak}.

We then have:
\begin{cor}\label{affine}
In Theorems \ref{expmixing} and \ref{expbmsmixing}, the exponents $\eta$ and $\beta$ can be chosen uniformly
over all congruence covers $\T^1(\Gamma_q\ba \bH^{d+1})$, $\Gamma_q\in \mathcal F$.
\end{cor}

When $\Gamma$ is convex cocompact, Corollary \ref{affine} was obtained by the second-named author and Winter 
\cite{OhWinter} for $d=1$ and by Sarkar \cite{Sa}
for a general $d\ge 1$, combining Dolgopyat's methods with the expander theory (\cite{BGS2}, \cite{SV}).

Theorems \ref{expmixing}-\ref{expbmsmixing} and Corollary \ref{affine} are known to have many immediate applications in number theory and geometry. To name a few, see  (\cite{DRS}, \cite{EM}, \cite{BO}, \cite{BGS}, \cite{MoOh}) for effective
 counting and affine sieve, \cite{MMO} for the prime geodesic theorem, and \cite{KelmerOh} for shrinking target problems.

\medskip

Fix $o\in \bH^{d+1}$ and $v_o\in \T^1(\bH^{d+1})$ based at $o$. Setting $K=\op{Stab}_G(o)\simeq \SO(d+1)$ and $M=\op{Stab}_G(v_o)\simeq \SO(d)$, we can identify $\Mm=\Gamma\ba G/K$
and $\T^1(\Mm)=\Gamma\ba G/M$. We let $\{a_t\}$ denote the 
one-parameter diagonalizable subgroup of $G$ commuting with $M$ whose right translation on $\Gamma\ba G/M$ 
corresponds to the geodesic flow $\mathcal G^t$ on $\T^1(\Mm)$. As $L^2(\T^1(\Mm))$ can be identified with the space of $M$-invariant functions in $L^2(\Gamma\ba G)$,
Theorem \ref{main} amounts to understanding 
the asymptotic behaviour of the matrix coefficients $\langle a_t \psi_1, \psi_2\rangle$ for $M$-invariant functions
$\psi_1, \psi_2\in L^2(\Gamma\ba G)$.

The non-tempered part of the unitary dual $\hat G$ consists of the trivial representation and the complementary series representations $\scrU(\upsilon,s)$, parameterized by a representation $\upsilon $ in the unitary dual $ \hat M$ and a real number
$s\in \scrI_{\upsilon}$, where $\scrI_{\upsilon}\subset (d/2, d)$ is an interval depending on $\upsilon$. In this parameterization, the complementary series representation $\scrU(\upsilon,s)$ is spherical if and only if $\upsilon\in \hat M$ is trivial, and for $\upsilon$ non-trivial, $\scrI_{\upsilon}$ is contained in $(d/2, d-1)$  \cite{Hirai}; this was the main reason
for the hypothesis $\delta>d-1$ in \cite{MoOh}.

Our main work for the proof of Theorem \ref{main} lies in the detailed analysis of the behavior of matrix coefficients 
 of the complementary series representations $\scrU(\upsilon,s)$.
  We remark that Harish-Chandra's work (\cite{Warner1}, \cite{Warner2}) does not give an asymptotic expansion
 for the matrix coefficients of $\scrU(\upsilon,s)$ for all $s$, as it excludes finitely many (unknown) parameters $s$. Even for those $\scrU(\upsilon, s)$ for which Harish-Chandra's expansion is given, it is hard to use his expansion directly, as it relies on various parameters and recursive formulas.
  
Denoting by $\langle \cdot, \cdot \rangle_{\scrU(\upsilon,s)}$ the {\it unitary} inner product, for each $m\in \NN$ we define a Sobolev norm $\|\cdot\|_{ {\scrS^m(\upsilon,s)}}$ on $\scrU(\upsilon,s)$ as in \eqref{sobnormdef}, and let $\scrS^m(\upsilon,s)$ denote the space of vectors in $\scrU(\upsilon,s)$ with finite $\|\cdot\|_{ {\scrS^m(\upsilon,s)}}$-norm.

Write $\scrU(\upsilon,s)=\oplus_{\tau\in \hat K} \; \scrU(\upsilon,s)_\tau $ for the decomposition into different $K$-types and for $d/2<s < d$, set  $$\eta_s:=\min \{2s-d, 1\}>0.$$

We show the following {\it concrete} asymptotic expansion of matrix coefficients with an optimal rate.
\begin{thm}\label{trivpropo}\label{Umatrixcoeffs0} 
There exists $m\in\NN$ such that for any complementary series representation $\scrU(\upsilon,s)$ containing 
a non-trivial $M$-invariant vector, for all $\vecu,\vecv\in \scrS^m(\upsilon,s)$ 
 and $t\geq 0$,
\begin{align*}
\langle \scrU(\upsilon, s) (a_t)\vecu,\vecv\rangle_{\scrU(\upsilon,s)} =&e^{(s-d)t}\left(\sum_{\tau_1,\tau_2\in{\hat K}}\langle T_{\tau_1}^{\tau_2}C_+(s)\mathsf{P}_{\tau_1}\vecu,\mathsf{P}_{\tau_2}\vecv\rangle_{\scrU(\upsilon,s)}\right)
\\&
\qquad\qquad+O_s\big(e^{(s-d-\eta_s)t}\|\vecu\|_{\scrS^m(\upsilon,s)}\|\vecv\|_{\scrS^m(\upsilon,s)} \big)
\end{align*}
 and the sum
\begin{equation}\label{a1}
\sum_{\tau_1,\tau_2\in{\hat K}}\langle T_{\tau_1}^{\tau_2}C_+(s)\mathsf{P}_{\tau_1}\vecu,\mathsf{P}_{\tau_2}\vecv\rangle_{\scrU(\upsilon,s)}
\end{equation}
converges absolutely. Here $ T_{\tau_1}^{\tau_2}: \scrU(\upsilon,s)_{\tau_1} \to  \scrU(\upsilon,s)_{\tau_2}$
 is given by \eqref{ttt}, $C_+(s)$ is the Harish-Chandra $c$-function given in \eqref{ccc}, and
 $\mathsf{P}_{\tau}$ is the orthogonal projection onto the $K$-type $\tau$. Moreover, the implied constant is uniformly bounded over $s$ in compact subsets of the interval $\scrI_{\upsilon}$.
\end{thm}

One of the key observations made in this paper is that if $ \scrU(\upsilon,s)$ is non-spherical and $\vecu\in \scrU(\upsilon,s)$ is $M$-invariant, then the main term \eqref{a1} vanishes (Corollary \ref{sigmazero}):
 for all $\tau_1, \tau_2\in \hat K$,
 $$ T_{\tau_1}^{\tau_2}C_+(s)\mathsf{P}_{\tau_1}\vecu=0.$$ 
Furthermore, there is additional uniformity with respect to the $s$-variable when considering only $M$-invariant vectors.

\begin{cor}\label{Mi}
There exists $m\in\NN$ such that if $\scrU(\upsilon,s)$ is non-spherical, then for all $M$-invariant vectors $\vecu,\vecv\in \scrU (\upsilon,s)$,
\begin{equation*}
 |\langle \scrU(\upsilon, s) (a_t)\vecu,\vecv\rangle_{\scrU(\upsilon,s)}|\ll_s e^{(s-d-\eta_s)t}\|\vecu\|_{\scrS^m(\upsilon,s)}\|\vecv\|_{\scrS^m(\upsilon,s)}.
\end{equation*}
Moreover, the implied constant is uniformly bounded over $s$ in compact subsets of $(\frac{d}{2},d)$.
\end{cor}

\noindent{\bf Organization.} We start by recalling in Section \ref{prelims} some key background facts and notation. In Section \ref{compseriessec}, we define the models of the complementary series that we will work in and the Eisenstein integrals that are key to understanding their matrix coefficients. The main technical work is done in Section \ref{matrixcoeffs}; establishing the asymptotic expansion and bounds for matrix coefficients in complementary series representations. Section \ref{Sobmixingsec} is devoted to extending Roblin's mixing result \cite[Theorem 3.4]{Roblin} from compactly supported functions to arbitrary Sobolev functions. Finally, the proof of Theorem \ref{expmixing} is given in Section \ref{thmproof}, the main part of which consists of decomposing the regular representation of $G$ on $L^2(\GaG)$ into irreducible representations and applying the bounds obtained in Section \ref{matrixcoeffs} to each individual component.


\section{Preliminaries}\label{prelims}

Let $d\ge 1$ and $G=\SO^\circ(d+1, 1)$ be the group of orientation-preserving isometries of $\bH^{d+1}$. Let  $\Gamma <G$ be a torsion-free
discrete subgroup of $G$, and let $\Mm=\Gamma\ba \bH^{d+1}$ be the associated hyperbolic manifold.
\subsection{Structure and subgroups of $G$.} 
Let $K=\SO(d+1) <G$ be a maximal compact subgroup, and let $A=\{a_t:t\in \br\}$ be a one-parameter diagonalizable subgroup
and $M=\SO(d)$ the centralizer of $A$ in $K$. 
We can identify $\Mm$ with $\Gamma\ba G/K$ and the unit tangent bundle $\T^1(\Mm)$ with
$\Gamma\ba G/M$  in the way that the geodesic flow on $\T^1(\Mm)$ is given by
 the right translation action of $a_t$ on $\Gamma\ba G/M$.

We denote by $N$ and $\overline N$ the contracting and expanding horospherical subgroups, respectively; i.e.\
$$N=\{g\in G: a_{-t} g a_{t}\to e\text{ as $t\to +\infty$}\};$$
$$\overline
N=\{g\in G: a_{t} g a_{-t}\to e\text{ as $t\to +\infty$}\}.$$
Then  $\overline{N}=\omega N\omega^{-1}$, where $\omega\in N_K(A)$ is such that $\omega a \omega^{-1}=a^{-1}$ for all $a\in A$.

We note that $N$ and $\overline{N}$  are both abelian subgroups, isomorphic to $\br^d$ via the logarithm map,
which are normalized by $AM$. Under this isomorphism $N\simeq \br^d$, conjugation by an element $m\in M=\SO(d)$ is an isometry  and conjugation by $a_t$ corresponds to scaling by $e^{t}$ on $N$, and $e^{-t}$ on $\overline{N}$.
As usual, the Lie algebras of $G$, $K$, $A$, $N$, $\overline N$ are denoted by $\fg$, $\fk$, $\fa$, $\fn$ and $\overline \fn$, respectively. 
This gives
\begin{equation*}
\Ad({a_t})|_{\fn}=e^t \times \mathrm{Id}\quad \text{ and } \quad \Ad({a_t})|_{\overline{\fn}}=e^{-t} \times \mathrm{Id}.
\end{equation*}

We have an Iwasawa decomposition $G=KAN$, where the product map $K\times A\times N\to G$ is a diffeomorphism and
the projection to each individual factor is a  smooth map \cite[Chapter VI.4]{Knapp2}. 
For $g\in G$, we write \begin{equation*}
g=\kappa(g)\exp(H(g)) n_g,
\end{equation*}
where $\kappa(g)\in K$, $H(g)\in \fa$, and $n_g\in N$. 
Letting $\alpha\in \fa_\CC^*$ be the unique element such that 
 $$\alpha (H(a_t))=t,$$ the map  $s\mapsto s\cdot\alpha$
 defines an identification $\CC \simeq \fa_{\CC}^*$. In view of this, we write
\begin{equation*}
e^{sH(g)}=e^{s\cdot\alpha(H(g))} \quad\text{for  $g\in G$ and $s\in \CC$.}
\end{equation*}

\subsection{ Various measures on $\T^1(\Mm)$.}
Let $\Lambda \subset \partial\bH^{d+1}$ denote the limit set of $\Gamma$. We assume that $\Gamma$ is non-elementary, or equivalently,
$\Lambda$ has at least $3$ points. Throughout the paper, we assume that
$${\textit{$\Gamma$ is  geometrically finite}},$$
 that is, the unit neighborhood of the convex core of $\Mm$, given by $\Gamma\ba \text{hull}(\Lambda)$, has finite Riemannian
 volume. We let $\delta$ denote the critical exponent of $\Gamma$, which is known to be equal to the Hausdorff dimension of $\Lambda$. We remark that $\delta=d$ if and only if $\partial \bH^n=\Lambda$ if and only if $\Gamma$ is a lattice in $G$.
 
We fix $o\in \bH^{d+1}$ whose stabilizer is given by $K$, and $v_o\in \T_o(\bH^{d+1})$ the unit vector whose stabilizer is $M$.
Let $\nu_o$ be the Patterson-Sullivan measure on $\Lambda$
which is unique up to a constant multiple; this is characterized by the condition
that $$\gamma^*\nu_o=|\gamma'|^{\delta}\cdot \nu_o$$ for all $\gamma\in \Gamma$, where $|\gamma'|$ denotes
 the derivative of $\gamma$ in the spherical metric on $\partial \bH^{d+1}$ with respect to $o$.
We also let $m_o$ denote the $K$-invariant probability measure on $\partial \bH^{d+1}$.

Let $\pi:\T^1(\bH^{d+1})\to \bH^{d+1}$ be the base point projection. For $u\in \T^1(\bH^{d+1})$, we denote by $u^{\pm}\in \partial \bH^{d+1}$ the forward and the backward endpoints of the geodesic determined by
$u$. Consider the Hopf parameterization of $\T^1(\bH^{d+1})$ given by:
\[
u \mapsto (u^+, u^-, s=\beta_{u^-} (o,\pi(u))),
\]
where  $\beta$ denotes the Busemann function.
Using the Hopf coordinates, the Bowen-Margulis-Sullivan measure $\bms$, 
 the Liouville measure $du$, and the Burger-Roblin measure $\m^{\rm{BR}}$ on $\T^1(\bH^{d+1})$ are respectively given as follows:
\begin{enumerate}
 \item 
$
d\bms(u)= e^{\delta \beta_{u^+}(o,
\pi(u))}\;
 e^{\delta \beta_{u^-}(o, \pi(u)) }\; d\nu_o(u^+) d\nu_o(u^-) ds;
$

\item  $ du=dm^{\op{Liouville}}(u)= e^{d\beta_{u^+}(o,
\pi(u))}\;
 e^{d\beta_{u^-}(o, \pi(u)) }\; dm_o(u^+) dm_o(u^-) ds;
$

\item
$ d\m^{\rm{BR}}(u)= e^{d \beta_{u^+}(o,
\pi(u))}\;
 e^{\delta \beta_{u^-}(o, \pi(u)) }\; dm_o(u^+) d\nu_o(u^-) ds.
 $
\end{enumerate}

The BR measure $m^{\BR}$ is a Lebesgue measure on each $\overline N$-leaf, so we call it the unstable BR measure.
The stable BR-measure $m^{\BR_*}$ is defined similarly to $m^{\BR}$ by exchanging the roles of $u^+$ and $u^-$.

These measures are all left $\Gamma$-invariant, and hence induce Borel measures on $\T^1(\Mm)=\Gamma\ba \T^1(\bH^{d+1})$
for which we use the same notation.
We remark that $m^{\BMS}$ is a finite measure, invariant under the geodesic flow. We will normalize the Patterson-Sullivan measure $\nu_o$ so that $m^{\BMS}(\T^1(\Mm))=1$. The other three measures are infinite, unless
$\delta=d$.

We will sometimes consider these measures as measures on $\Gamma\ba G$ by putting: for $\psi\in C_c(\Gamma\ba G)$,
and for $\star=\BMS,\op{Liouville}, \BR$,
$$m^{\star}(\psi)=m^{\star} \left(\int_M \psi \; dm\right),$$ 
where $dm$ denotes the Haar probability measure on $M$. Note that the Liouville measure,
considered as a measure on $\Gamma\ba G$, is  $G$-invariant; we will denote this by $dg$.

\subsection{ The base eigenfunction $\phi_0$.}
Throughout the article,  we assume
 $$\delta >\frac{d}{2} ,$$
which is a necessary and sufficient condition for the existence of an eigenvalue of $\Delta$ on $L^2(\Mm)$. The smallest eigenvalue of $\Delta$ on $L^2(\Mm)$ is given by $\la_0=\delta (d-\delta)$, and is known to be simple. We denote by $\phi_0$ the the unit eigenfunction in $L^2(\Mm)$ with eigenvalue $\lambda_0$, and call it the \emph{base eigenfunction}. Up to a constant multiple, $\phi_0$ is given by: for $x\in \Mm$,
$$\phi_0(x)=\int_{\xi\in \Lambda} e^{-\delta \beta_{\xi} (o, x)} d\nu_o(\xi).$$
 
The fact that the base eigenfunction $\phi_0$ is square-integrable when $\delta>d/2$ is a key result,
 which allows the unitary representation theory of $G$, specifically the right translation action of $G$ on
 $L^2(\Gamma\ba G)$,  to be used  to prove dynamical results for the geodesic flow.

\begin{thm} [Lax-Phillips \cite{LaxPhillips}] \label{spectralgap}
The intersection of the interval $[0,\frac{d^2}{4})$ with the spectrum of $\Delta$, viewed as an unbounded operator on $L^2(\scrM)$, consists of
a finite set of eigenvalues $\lbrace \lambda_i=s_i(d-s_i)\rbrace_{0\leq i\leq \ell}$, satisfying
\begin{equation*}
0<\lambda_0=\delta(d-\delta)<\lambda_1 \leq \ldots \leq\lambda_\ell < \frac{d^2}{4}.
\end{equation*}
\end{thm}

\subsection{The quasi-regular representation $L^2(\GaG)$.}\label{repdefsec}
We denote $L^2(\GaG)$ the space of square-integrable functions on $\Gamma\ba G$ with respect to $dg$.
The $G$-invariance of $dg$ gives rise to a unitary representation $\big(\rho,L^2(\GaG)\big)$ of $G$, where $\rho$ is given by the right translation:
\begin{equation*}
[\rho(g)f](x)=f(xg)
\end{equation*}
for all $f\in L^2(\GaG),\, g\in G,\, x\in \GaG.$

As a number of inner products on different vector spaces will show up throughout the article, we reserve now $\langle\cdot,\cdot\rangle$ to mean the $\rho(G)$-invariant inner product on $L^2(\GaG)$. All other inner products will have some additional notation to distinguish them.
The subspace of $\rho(K)$-invariant vectors in $L^2(\GaG)$ is denoted $L^2(\GaG)^K$. Similarly, $L^2(\GaG)^M$ denotes the subspace of $\rho(M)$-invariant vectors. We use $L^2(\GaG)^K$ to construct subrepresentations of $\big(\rho,L^2(\GaG)\big)$ as follows: define
\begin{equation*}
L^2(\GaG)_{\mathrm{sph}}:=\text{the closure of }{\lbrace \rho(g)f\,:\, f\in L^2(\GaG)^K,\,g\in G\rbrace}.
\end{equation*}
Then 
\begin{equation*}
\big(\rho,L^2(\GaG)\big)= \big(\rho,L^2(\GaG)_{\mathrm{sph}}\big)\bigoplus \big(\rho,L^2(\GaG)_{\mathrm{sph}}^{\perp}\big);
\end{equation*}
note that both $L^2(\GaG)_{\mathrm{sph}}$ and $L^2(\GaG)_{\mathrm{sph}}^{\perp}$ are $\rho(G)$-invariant closed subspaces.
 Viewing the base eigenfunction $\phi_0$ as an element of $L^2(\GaG)^K$, we define in a similar manner
\begin{equation*}
\scrB_{\delta}:=\text{the closure of the span of }\{\rho(g)\phi_0\,:\, g\in G\}\subset L^2(\GaG)_{\mathrm{sph}},
\end{equation*}
and
\begin{equation*}
 \big(\rho,L^2(\GaG)_{\mathrm{sph}}\big)=(\rho,\scrB_{\delta})\oplus(\rho,\scrW),
\end{equation*}
$\scrW$ being the orthogonal complement of $\scrB_{\delta}$ in $L^2(\GaG)_{\mathrm{sph}}$. It will be of importance later that $(\rho,\scrB_{\delta})$ is an (irreducible) complementary series representation and that Theorem \ref{spectralgap} gives a complete understanding of the complementary series representations contained in $(\rho,\scrW)$. A useful fact (cf. \cite{KontOh}) that we make of in Section \ref{Sobmixingsec} is that
for all $f\in L^2(\GaG)^K$,
\begin{equation}\label{mBRinnerprod}
\mBR(f)=\mBRast(f)= \langle f,\phi_0\rangle.
\end{equation}

Finally, the direct integral decomposition  of $\big(\rho,L^2(\GaG)\big)$ reads
\begin{equation}\label{L2decomp}
\big(\rho,L^2(\GaG)\big)\cong \int_{\mathsf{Z}}^{\oplus} \big(\mathcal{\pi}_{\zeta}, \mathcal{H}_{\zeta}\big)\,d\mu_{\mathsf{Z}}(\zeta),
\end{equation}
where $(\pi_{\zeta}, \mathcal{H}_{\zeta})$ is an irreducible unitary representation of $G$ for $\mu_{\mathsf{Z}}$-a.e.\ $\zeta$ (cf.\ \cite[Corollary 14.9.5]{Wallach}). 
\subsection{Sobolev norms on unitary representations of $K$}\label{Sobnormdefsec}
Given a unitary representation $(\pi,V)$ of $K$ with invariant inner product $(\cdot,\cdot)_V$,
a basis $\lbrace X_j\rbrace$ of $\mathfrak{k}$ and $m\in\NN$, we define a Sobolev norm $\|\cdot\|_{\scrS^m(V)}$ on $V$ by
\begin{equation}\label{sobnormdef}
\|\vecv\|_{\scrS^m(V)}^2:=\|\vecv\|^2_V+\sum_{U} \| d\pi(U)\vecv\|_V^2\qquad \text{ for any } \vecv\in V,
\end{equation}
where $\|\cdot\|_V$ denotes the corresponding norm and the sum runs over all monomials in $\{X_j\}$ of order up to $m$. 

We set $\scrS^m(V):=\{\vecv\in V:\|\vecv\|_{\scrS^m(V)}<\infty\}$. Observe that different choices of the basis $\lbrace X_j\rbrace$ give rise to equivalent norms, and that in the case when $V$ is finite-dimensional, we have $\scrS^m(V)=V$ for any $m\ge 0$. 

Viewing $\big(\rho,L^2(\GaG)\big)$ as a unitary representation of $K$, given a function $f\in L^2(\GaG)$, we let either $\|f\|_{\scrS^m(\GaG)}$ or simply $\scrS^m(f)$ denote the norm $\|f\|_{\scrS^m(L^2(\GaG))}$ as defined by \eqref{sobnormdef} above.
We denote by $\scrS^m(\GaG)$ the space of all functions $f\in L^2(\GaG)$ with $\scrS^m(f)<\infty$.

For $f\in L^2(\Gamma\ba G)$, $m^{\BR}(f)$ may be infinite in general when $\delta<d$. However we have the following lemma: \begin{lem}\label{mMeasSobbd}
If $m>\frac{(d+1)d}{4}$, then for any $f\in\scrS^m(\GaG)$,  
\begin{equation*}
m^{\BR}(f),\; m^{\BR_*}(f) \; \ll\; \scrS^{m}(f).
\end{equation*}
\end{lem}
\begin{proof} 
Given $f\in C(\GaG)\cap L^2(\GaG)$, define
\begin{equation*}
f_K( x):= \max_{k\in K} |f(xk)| \qquad \text{ for } x\in\GaG.
\end{equation*}
By construction, $f_K$ is $\rho(K)$-invariant and $ |f(x)|\leq f_K(x)$ for all $ x\in \GaG$, hence $|m^{\BR}(f)|\leq m^{\BR}(f_K)$. 
In view of \eqref{mBRinnerprod},  the Sobolev embedding theorem on the compact manifold $K$ gives
\begin{equation*}
|m^{\BR}( f )| \leq m^{\BR}(f_K)=\langle f_K,\phi_0\rangle\leq \|f_K\|\ll \|f\|_{\scrS^m(\GaG)}
\end{equation*}
for all $m>\frac{\dim(K)}{2}$ (cf.\ \cite[Theorem 2.30]{Aubin}). Since $\dim K=\frac{d(d+1)}2$, the chain of inequalities above holds for any $m>\frac{(d+1)d}{4}$.
The claim for $m^{\BR_*}$ follows similarly.
\end{proof}

\section{Complementary series representations and Eisenstein integrals}\label{compseriessec}
In this section we recall the definition and $K$-type structure of the complementary series representations of $G=\SO^\circ(d+1, 1)$, as well as the Eisenstein integral representations of matrix coefficients for these representations. We start by reviewing the representation theory of
 the special orthogonal groups $\SO(n)$.
\subsection{Irreducible representations of $K$ and $M$}\label{SO(n)reps}
The primary reference for the first facts listed below is \cite[pp. 272-277]{BrtD}. Let $m\ge 1$.
The irreducible representations of $\SO(2m)$ are parameterized by $m$-tuples $\tau=(\tau_1,\tau_2,\ldots,\tau_m)\in\ZZ^{m}$ such that
\begin{equation*}
\tau_1\geq\tau_2\geq\ldots\geq \tau_{m-1}\geq \tau_{m-1}\geq |\tau_m|.
\end{equation*}
Similarly, the irreducible representations of $\SO(2m+1)$ are parameterized by $m$-tuples $\tau=(\tau_1,\tau_2,\ldots,\tau_m)\in\ZZ^{m}$ such that
\begin{equation*}
\tau_1\geq\tau_2\geq\ldots\geq \tau_{m-1}\geq \tau_{m}\geq 0.
\end{equation*}
We recall that all irreducible representations of $\SO(2m+1)$ and $\SO(4m)$ are self-dual, and that the dual $\tau^*$ of an irreducible representation $\tau=(\tau_1,\tau_2,\ldots,\tau_{2m+1})$ of $\SO(4m+2)$ is given by $\tau^*=(\tau_1,\tau_2,\ldots,-\tau_{2m+1})$; the dual $\tau^*$ is defined by $\tau^*(k)=\tau(k^{-1})$, acting on the dual (or complex conjugate) of the underlying vector space.

A key result that we will make repeated use of is the branching law for restrictions of representations of $K$ to $M$ which we recall
 (cf.\ \cite[Chapter IX.3]{Knapp2} or \cite[pp. 377-380]{Zelo}). We henceforth call irreducible representations of $K$ or $M$ a $K$ or $M$-type, respectively. 
 Let $\tau$ be a $K$-type.  Then the decomposition of $\tau|_M$ reads
\begin{equation*}
\tau|_M=\bigoplus_{\sigma\in{\hat M}} m_{\sigma,\tau}\,\cdot\, \sigma,
\end{equation*}
where $m_{\sigma,\tau}=1$ if $\sigma=(\sigma_1,\ldots,\sigma_{\lfloor \frac{d}{2}\rfloor})$ satisfies the interlacing property
\begin{equation*}
\tau_1\geq \sigma_1\geq \tau_2\geq \sigma_2\geq\ldots \geq \sigma_{m-1}\geq |\tau_m| \qquad\mathrm{for\;} d=2m-1,
\end{equation*}
and
\begin{equation*}
\tau_1\geq \sigma_1\geq \tau_2\geq \sigma_2\geq\tau_m\geq |\sigma_m| \qquad\mathrm{for\;} d=2m,
\end{equation*}
and $m_{\sigma,\tau}=0$ otherwise. If $m_{\sigma,\tau}=1$, we say that $\tau$ contains $\sigma$, and write $\sigma\subset \tau$. The fact that each $M$-type occurs at most once in $\tau$ will allow us to make repeated use of Schur's lemma in the following manner: if a linear operator on $V_{\tau}$ ($V_{\tau}$ being a vector space on which $\tau$ is realized) commutes with $\tau(M)$, then it acts as a scalar on each $\sigma\subset\tau$. From now on we write $\dim(\tau)$ and $\dim(\sigma)$ for $\dim(V_{\tau})$ and $\dim(V_{\sigma})$, respectively.

We connect the dimensions of $K$-types with the Sobolev norms introduced in Section \ref{Sobnormdefsec}. Firstly, let $(\pi,\scrH)$ be a unitary representation of $G$. For each $\tau\in{\hat K}$, we define
$$\chi_{\tau}(k)=\dim (V_{\tau})\,\cdot\, \mathrm{tr}\big(\tau(k)\big)$$ (the trace being defined with respect to any invariant inner product on any realization of $\tau$)
and
 an operator $\mathsf{P}_{\tau}$ by
\begin{equation*}
\mathsf{P}_{\tau}=\int_K \overline{\chi_{\tau}(k)}\pi(k)\,dk.
\end{equation*}
 Note that $\mathsf{P}_{\tau}$ is the orthogonal projection onto the space $\scrH_{\tau}$, where
\begin{equation*}
\scrH_{\tau}:=\lbrace \vecv\in\scrH\,:\, \mathsf{P}_{\tau}\vecv=\vecv\rbrace.
\end{equation*}
This gives rise to a decomposition of $\scrH$ as the orthogonal direct sum 
\begin{equation*}
\scrH=\bigoplus_{\tau\in{\hat K}} \scrH_{\tau}.
\end{equation*}
If $(\pi,\scrH)$ is irreducible then each $\scrH_{\tau}$ is finite-dimensional. Each $\scrH_{\tau}$ has the property that $\pi(k)|{\scrH_{\tau}}\cong \tau(k) $ for all $k\in K$.  If $\scrH_{\tau}\neq 0$, we say that $\pi$ contains $\tau$, and write $\tau\subset \pi$. There is a similar decomposition of $\scrH$ with respect to $M$-types, and projection operators $\mathsf{P}_{\sigma}=\int_M \overline{\chi_{\sigma}(m)}\pi(m)\,dm$ for $\sigma\in \hat M$. As usual, $\lceil \beta \rceil $ denotes the smallest integer greater than or
equal to $\beta$.
\begin{lem}\label{Fourier}
There exists $m_0\in\NN$ depending only on $K$ such that for any unitary representation $(\pi,V)$ of $K$, $\alpha>0$, $m\in\NN$, and for all $\vecv\in V$, 
\begin{equation*}
\sum_{\tau\subset V} \dim(\tau)^{\alpha}\|\mathsf{P}_{\tau}\vecv\|_{\scrS^{m}(V)}\ll \|\vecv\|_{\scrS^{m+m_0\lceil \frac{\alpha +1}{2}\rceil}(V)},
\end{equation*}
where the implied constant depends only on $K$.
\end{lem}
\begin{proof}
 Though the argument is standard (cf.\ \cite[Chapter 4.4.2]{Warner1} or \cite[Lemmas 10.3 and 10.4]{Knapp1}), we briefly recount it: using the Harish-Chandra isomorphism and  the highest weight theory, we find an element $\omega_K$ in the center of the universal enveloping algebra of $\mathfrak{k}$ such that for each $\tau\in{\hat K}$, $d\tau(\omega_K)$ acts as a scalar $c_{\tau}$ on $\tau$, where $|c_{\tau}|\geq \dim(\tau)^2$. Letting $m_0$ be the order of $\omega_K$, for any $\ell\in \NN$, we have (using the Cauchy-Schwarz inequality),
\begin{align*} 
 \sum_{\tau\subset V} \dim(\tau)^{\alpha}\|\mathsf{P}_{\tau}\vecv\|_{\scrS^{m}(V)}&=\sum_{\tau\subset V} \frac{\dim(\tau)^{\alpha}}{|c_{\tau}|^\ell}\|\mathsf{P}_{\tau}d\pi(\Omega_K^\ell)\vecv\|_{\scrS^{m}(V)}
 \\ & \le \left( \sum_{\tau\subset V} \dim(\tau)^{2\alpha-4\ell}\sum_{\tau\subset V} \|\mathsf{P}_{\tau}d\pi(\Omega_K^\ell)\vecv\|^2_{\scrS^{m}(V)}\right)^{1/2}\\&\ll  \left(\sum_{\tau\subset V} \dim(\tau)^{2\alpha-4\ell}\right)^{1/2}  \|\vecv\|_{\scrS^{m+\ell m_0}(V)}.
\end{align*} 
The sum $\sum_{\tau} {\dim(\tau)^{-2}}$ is finite by \cite[Lemma 13]{Knapp1}, so choosing $\ell=\lceil \frac{ \alpha +1}{2}\rceil$ gives $\sum_{\tau\subset V} \dim(\tau)^{2\alpha-4\ell} <\infty$.
\end{proof}
\subsection{Complementary series representations}\label{compserdefsec}
We will now recall Hirai's classification \cite{Hirai} of the non-tempered unitary dual of $G$ and construct the models of the complementary series that we will work with, see \cite[Chapter 5.5]{Warner1}  (cf.\ also \cite[Sections 3.1-3.3]{MoOh}). 

Given $s\in\CC$, we define the standard representation $U^s$ of $G$ on $L^2(K)$ by
\begin{equation}\label{uss}
[U^s(g)\vecv](k)=e^{-sH(g^{-1}k)}\vecv\big(\kappa(g^{-1}k)\big)
\end{equation}
 for all $\vecv\in L^2(K),\; g\in G,$ and $ k\in K$.
We let $\lambda$ and $\rho$ denote the left- and right-regular representations of $K$ on $L^2(K)$ respectively. Observe that $U^s|_K=\lambda$; for this reason, it is practical to always view $L^2(K)$ as the unitary representation $(\lambda, L^2(K))$ of $K$. The decomposition of $L^2(K)$ into $K$-types reads
\begin{equation*}
L^2(K)=\bigoplus_{\tau\in{\hat K}} L^2(K)_{\tau},
\end{equation*}
where each $L^2(K)_{\tau}$ is isomorphic to $\dim(\tau)$ copies of $\tau$. Given $\upsilon\in{\hat M}$, we define
\begin{align*}
L^2(K:\upsilon):=\left\lbrace \vecv\in L^2(K)\,:\, \int_M\overline{\chi_{\upsilon}(m)}\rho(m)\vecv\,dm=\vecv\right\rbrace.
\end{align*}

The fact that $U^s(g)$ and $\rho(m)$ commute for all $g\in G$ and $m\in M$ shows that $\big(U^s,L^2(K:\upsilon) \big)$ is a representation of $G$; in fact $\big(U^s,L^2(K:\upsilon) \big)$ is isomorphic to $\dim(\upsilon)$ copies of the representation
$\mathrm{ind}_{MAN}^G ( \upsilon\otimes s\otimes 1)$ (cf.\ \cite[Chapter VII]{Knapp1}). Note that $L^2(K:\upsilon)_{\tau}\subset L^2(K)_{\tau}$ for all $\tau\in{\hat K}$. From Hirai's classification of the unitary dual of $G$ \cite{Hirai}, for each $\upsilon\in {\hat M}$, there exists an interval $\scrI_{\upsilon}\subset (\frac{d}{2},d)$ such that if $s\in \scrI_{\upsilon}$, then there exists an irreducible, unitarizable subrepresentation $\scrU(\upsilon,s)$ of $\big(U^s,L^2(K:\upsilon) \big)$. Furthermore, every non-tempered representation of $G$ may be realized as  $\scrU(\upsilon,s)$
for some $\upsilon\in \hat M$ and $s\in \scrI_{\upsilon}$ in this way, and
each $\tau\in \hat K$ that contains $\upsilon$ occurs exactly once in $\scrU(\upsilon,s)$
(cf.\ \cite[Theorem 5.5.1.5]{Warner1} and \cite[Theorem 8.37]{Knapp1}). 

The inner product that makes $\scrU(\upsilon,s)$ a unitary representation is denoted by
$\langle \cdot,\cdot\rangle_{\scrU(\upsilon,s)}$. Note that for $g\in G$, and $\vecv, \vecu \in \scrU(\upsilon,s)$,
$$\langle U^s(g) \vecv, \vecu\rangle_{\scrU(\upsilon,s)}=\langle \scrU(\upsilon,s)(g) \vecv, \vecu\rangle_{\scrU(\upsilon,s)} .$$

We conclude this section by recalling two general facts regarding complementary series representations that let us classify the representation $\scrB_{\delta}$ (cf. Section \ref{repdefsec}). Firstly, the spherical complementary series representation are the $\scrU(1,s)$; i.e.\ precisely those where the representation $\upsilon$ is trivial. Secondly, the Casimir operator $\scrC$ of $G$ acts as the scalar $-s(d-s)$ on the smooth vectors of $\scrU(\upsilon,s)$. Combining these facts with the observation that the restriction of $\scrC$ to right $K$-invariant functions on $G$ is the Laplace-Beltrami operator lets one conclude that the representation $\scrB_{\delta}$ is isomorphic to $\scrU(1,\delta)$.

\subsection{Eisenstein integrals}
Here we develop the Eisenstein integrals needed to obtain the desired asymptotic expansion of matrix coefficients
(cf.\ \cite[Section 3.3]{MoOh}, \cite[Chapter 6]{Warner2}). Before starting, we make the following elementary but important observation:
\begin{lem}\label{Mzerolem} If $\sigma\in{\hat M}\setminus\lbrace\upsilon^*\rbrace$, then
for any $\vecv\in L^2(K:\upsilon)_{\sigma}$,  we have
$$\vecv(m)=0\quad\text{ for all $m\in M$.}$$
\end{lem}
\begin{proof} Since $\vecv\in L^2(K:\upsilon)$, we have $\vecv=\int_M \rho(m)\vecv \overline{\chi_{\upsilon}(m)}\,dm$. Hence for any $m\in M$,
\begin{align*}
\vecv(m)&=\int_M \vecv(mm_1)\overline{\chi_{\upsilon}(m_1)}\,dm_1\\
&=\int_M \vecv(m_1m)\overline{\chi_{\upsilon}(m^{-1}m_1m)}\,dm_1
\\&=\int_M \vecv(m_1^{-1}m)\overline{\chi_{\upsilon}(m_1^{-1})}\,dm_1\\&=\int_M \vecv(m_1^{-1}m)\overline{\chi_{\upsilon^*}(m_1)}\,dm_1
\\&=[\mathsf{P}_{\upsilon^*}\vecv](m).
\end{align*}
Since  $\vecv\in L^2(K:\upsilon)_{\sigma}$ for $\sigma\in{\hat M}\setminus\lbrace\upsilon^*\rbrace$,
 the orthogonality of the $M$-types of $L^2(K)$ gives $\mathsf{P}_{\upsilon^*}\vecv=0$, yielding the claim.
\end{proof}
Before stating the main result of this section, we introduce some more notation. Firstly, we let $\langle\cdot,\cdot\rangle_K$ denote the usual inner product on $L^2(K)$. The corresponding norm on $L^2(K)$ is denoted $\|\cdot\|_{K}$, and similarly for the operator norm defined with respect to it. 

In the rest of this section, we fix a complementary series representation $\scrU(\upsilon,s)$, $s\in \scrI_{\upsilon}$.
 For each $K$-type $\tau$ of $\scrU(\upsilon,s)$ we define a vector $\vecchi_{\tau}\in \scrU(\upsilon,s)_{\tau}$ by 
\begin{equation}\label{Projchi}
\vecchi_{\tau}=\sum_{i=1}^{\dim(\tau)} \overline{\vecv_i(e)}\,\vecv_i,
\end{equation}
where $\lbrace\vecv_i\rbrace_{i=1}^{\dim(\tau)}$ is an orthornormal basis of $\scrU(\upsilon,s)_{\tau}$ (recall that $\tau$ has multiplicity one in $\scrU(\upsilon,s)$). Note that $\vecchi_{\tau}$ is independent of the choice of basis, and  for all $ \vecv\in\scrU(\upsilon,s)_{\tau}$ and $ k\in K$, we have
\begin{equation}\label{evalinnerprod}
\vecv(k)=\langle \vecv, U^s(k)\vecchi_{\tau}\rangle_K.
\end{equation}

\begin{lem}\label{Mvint} Let $\tau$ be a $K$-type of $ \scrU(\upsilon,s)$. 
\begin{enumerate}
\item For all $\vecv\in \scrU(\upsilon,s)_{\tau}\cap \scrU(\upsilon,s)_{\upsilon^*}$, we have
\begin{equation*}
\int_M |\vecv(m)|^2\,dm=\frac{\dim(\tau)\|\vecv\|_{K}^2}{\dim(\upsilon)}.\end{equation*}
\item $\int_M |\vecchi_{\tau}(m)|^2\,dm=\frac{\dim(\tau)^2}{\dim (\upsilon)}$.
\end{enumerate}
\end{lem}
\begin{proof}
Let $\lbrace \vecv_j\rbrace$ be an orthonormal basis  of $\scrU(\upsilon,s)_{\tau}\cap \scrU(\upsilon,s)_{\upsilon^*}$ with respect to $\langle \cdot,\cdot\rangle_K$. Using \eqref{Projchi} and
the Schur orthogonality relations, we get
\begin{align*}
\int_M |\vecv(m)|^2\,dm &=\int_M |\langle\vecv,U^s(m)\vecchi_{\tau} \rangle_K|^2\,dm
\\&=\sum_{i,j} \overline{\vecv_i(e)}\vecv_j(e)\int_M \langle U^s(m)\vecv,\vecv_i \rangle_K\overline{\langle U^s(m)\vecv,\vecv_j \rangle_K}\,dm
\\& =\frac{\|\vecv\|_{K}^2\sum_i |\vecv_i(e)|^2}{\dim(\upsilon)}.
\end{align*}
Complete $\lbrace \vecv_j\rbrace$ into an
orthonormal basis  $\lbrace \vecv_j\rbrace \cup \lbrace \vecw_i\rbrace $ for all of $\scrU(\upsilon,s)_{\tau}$.
Note that $\vecw_i(e)=0$ for all $i$  by Lemma \ref{Mzerolem}.
Therefore we have
\begin{equation}\label{Wmbasissum}
\sum_j |\vecv_j(e)|^2=\sum_j |\vecv_j(e)|^2+\sum_i |\vecw_i(e)|^2=\dim(\tau).
\end{equation} 
Hence the claim (1) follows. Claim (2) follows from (1).
\end{proof}

For $K$-types $\tau_1, \tau_2$ of $\scrU(\upsilon,s)$, we define an operator
$$ T_{\tau_1}^{\tau_2}: \scrU(\upsilon,s)_{\tau_1}\to \scrU(\upsilon,s)_{\tau_2}$$ by
\begin{equation}\label{ttt}
T_{\tau_1}^{\tau_2}\vecv=\int_M \vecv(m) U^s(m)\vecchi_{\tau_2}\,dm\qquad \text{for all } \vecv\in\scrU(\upsilon,s)_{\tau_1}.
\end{equation}

We have the following interpretation of the Eisenstein integral (a similar formula appears in \cite[Theorem 3.4]{MoOh}): \begin{thm}\label{Eisensteinthm}
For any $K$-types $\tau_1$, $\tau_2$ of $\scrU(\upsilon,s)$ and $g\in G$, we have
\begin{equation*}
\mathsf{P}_{\tau_2}U^s(g)\mathsf{P}_{\tau_1}=\int_K e^{(s-d)H(gk)} U^s\big(\kappa(gk)\big) T_{\tau_1}^{\tau_2} U^s(k^{-1})\,dk,
\end{equation*} 
where $\mathsf{P}_{\tau_2}U^s(g)\mathsf{P}_{\tau_1}$ is viewed as an operator from $\scrU(\upsilon,s)_{\tau_1}$ to $\scrU(\upsilon,s)_{\tau_2}$.
\end{thm}
\begin{proof}
Following \cite[Theorem 6.2.2.4, pp. 42-43]{Warner2}, the key fact is that for any $g\in G$,
$$dk=e^{d\, H(gk)}d \big(\kappa (gk)\big).$$
 Let $\vecv\in \Uus_{\tau_1}$ and $\vecw\in \Uus_{\tau_2}$. Then
\begin{align*}
\langle U^s(g)\vecv,\vecw\rangle_K &=\int_K e^{-sH(g^{-1}k)}\vecv\big(\kappa(g^{-1}k)\big)\overline{\vecw(k)}\,dk
\\ &=\int_K e^{-sH(g^{-1}k)}\vecv\big(\kappa(g^{-1}k)\big)\overline{\vecw(k)}\,e^{d\,H(g^{-1}k)}\,d\big(\kappa (g^{-1}k)\big).
\end{align*}
We now carry out the change of variables $\tilde{k}=\kappa (g^{-1}k)$; this gives $k=\kappa(g\tilde{k})$, allowing us to rewrite the above integral as follows
\begin{align*}
&=\int_K e^{(d-s)H\big(g^{-1}\kappa(gk)\big)}\vecv(k)\overline{\vecw\big(\kappa(gk)\big)}\,\,dk
\\&=\int_K e^{(s-d)H(gk)}\vecv(k)\overline{\vecw\big(\kappa(gk)\big)}\,dk,
\end{align*}
where we used the identity 
$$g^{-1}\kappa(gk)=k\exp(-H(gk)) \big(\exp(H(gk)n_{kg}^{-1}\exp(-H(gk)\big)$$
to obtain $H(g^{-1}\kappa(gk))=-H(gk)$.

By \eqref{evalinnerprod}, we get
\begin{align*}
\vecv(k)\overline{\vecw\big(\kappa(gk)\big)}=&\langle \vecv,U^s(k)\vecchi_{\tau_1}\rangle_K \cdot
 \overline{\langle \vecw,U^s\big(\kappa(gk)\big)\vecchi_{\tau_2}\rangle_K}
\\=&\langle U^s(k^{-1})\vecv,\vecchi_{\tau_1}\rangle_K \cdot \langle U^s\big(\kappa(gk)\big)\vecchi_{\tau_2},\vecw\rangle_K
\\=&\Big\langle\langle U^s(k^{-1})\vecv,\vecchi_{\tau_1}\rangle\,\cdot\,  U^s\big(\kappa(gk)\big)\vecchi_{\tau_2},\vecw\Big\rangle_K.
\end{align*}
Hence
\begin{align*}
\langle U^s(g)\vecv,\vecw\rangle_K=&\int_K e^{(s-d)H(gk)}\Big\langle\langle U^s(k^{-1})\vecv,\vecchi_{\tau_1}\rangle\,\cdot\,  U^s\big(\kappa(gk)\big)\vecchi_{\tau_2},\vecw\Big\rangle_K\,dk
\\=&\left\langle\int_K e^{(s-d)H(gk)}\langle U^s(k^{-1})\vecv,\vecchi_{\tau_1}\rangle\,\cdot\,  U^s\big(\kappa(gk)\big)\vecchi_{\tau_2}\,dk,\vecw\right\rangle_K.
\end{align*}
We define $T_0:\scrU(\upsilon,s)_{\tau_1}\rightarrow \scrU(\upsilon,s)_{\tau_2}$ by 
$$T_0\vecv= \langle \vecv , \vecchi_{\tau_1}\rangle_K \cdot  \vecchi_{\tau_2}=\vecv(e)\vecchi_{\tau_2}.$$ 
Then 
\begin{equation}\label{PUsP}
\mathsf{P}_{{\tau_2}}U^s(g)\mathsf{P}_{{\tau_1}}= \int_K e^{(s-d)H(gk)} U^s\big(\kappa(gk)\big) T_0U^s(k^{-1}) \,dk.
\end{equation} 
Using \eqref{uss} and \eqref{ttt}, we observe that \begin{equation*}
T_{\tau_1}^{\tau_2}=\int_M U^s(m) T_0 U^s(m^{-1})\,dm.
\end{equation*}
Writing $g=k_1 a_t k_2$, we then have
\begin{align*}
&\mathsf{P}_{{\tau_2}}U^s(g)\mathsf{P}_{{\tau_1}} = U^s(k_1) \mathsf{P}_{{\tau_2}}U^s(a_t)\mathsf{P}_{{\tau_1}} U^s(k_2) 
\\&=\int_M U^s(k_1) U^s(m)\mathsf{P}_{{\tau_2}}U^s(a_t)\mathsf{P}_{{\tau_1}}U^s(m^{-1}) U^s(k_2)\,dm,
\end{align*}
where we used the facts that $M=Z_K(A)$ and $M\subset K$ to commute $U^s(m)$ past $U^s(a_t)$ and $\mathsf{P}_{\tau_i}$. Now using \eqref{PUsP} gives 
\begin{align*}
&= U^s(k_1) \left(\int_M \int_K e^{(s-d)H(a_tk)} U^s\big(m\kappa(a_tk)\big)T_0U^s(k^{-1}m^{-1})\,dk\,dm\, \right)U^s(k_2) 
\\&=   \int_K e^{(s-d)H(a_tk)} U^s\big(k_1\kappa(a_tk)\big)\left( \int_M U^s(m)T_0 U^s({m^{-1}})\,dm\right) U^s(k^{-1}k_2)\,dk.
\end{align*}
Finally, using the identities $k_1\kappa(a_tk)=\kappa(k_1a_tk)$ and $H(a_tk)=H(k_1a_tk)$ together with the change of variables $k=k_2\tilde{k}$ gives
\begin{align*}
\mathsf{P}_{{\tau_2}}U^s(g)\mathsf{P}_{{\tau_1}}=   \int_K e^{(s-d)H(g\tilde{k})} U^s\big(\kappa(g\tilde{k})\big)T_{\tau_1}^{\tau_2} U^s(\tilde{k}^{-1})\,d\tilde{k},
\end{align*}
proving the first identity claimed in the theorem. The second identity follows from the fact that $\langle U^s(m^{-1})\vecv,\vecchi_{\tau_1}\rangle_K=\vecv(m)$ for all $\vecv\in \Uus_{\tau_1}$ and $m\in M$.
\end{proof}

\subsection{On the operator $T_{\tau_1}^{\tau_2}$.}
 Using Lemma \ref{Mzerolem}, we deduce:
 \begin{prop}\label{sigmazero}  Let $\tau_1$, $\tau_2$ be $K$-types of $\scrU(\upsilon,s)$.
If $\vecv\in \scrU(\upsilon,s)_{\tau_1}$ is orthogonal to $\scrU(\upsilon,s)_{\upsilon^*}$, then 
$$T_{\tau_1}^{\tau_2}\vecv=0.$$ 
Consequently, $T_{\tau_1}^{\tau_2}\, \scrU(\upsilon,s)_{\tau_1}\subset \scrU(\upsilon,s)_{\tau_2}\cap \scrU(\upsilon,s)_{\upsilon^*}$.
\end{prop} 
\begin{proof}
We have $T_{\tau_1}^{\tau_2}\vecv=\int_M \vecv(m) U^s(m)\vecchi_{\tau_2}\,dm$. By Lemma \ref{Mzerolem}, 
if $\vecv$ is orthogonal to $\scrU(\upsilon,s)_{\upsilon^*}$, then $\vecv(m)=0$ for all $m\in M$
and thus $T_{\tau_1}^{\tau_2}\vecv=0$. A direct computation using \eqref{ttt} shows that $T_{\tau_1}^{\tau_2}$ commutes with $U^s(m)$ for all $m\in M$, so Schur's lemma then gives that if $T_{\tau_1}^{\tau_2}\vecv$ is non-zero, it must be contained in $\scrU(\upsilon,s)_{\tau_2}\cap \scrU(\upsilon,s)_{\upsilon^*}$. 
\end{proof}

If $\vecv$ is $M$-invariant and $\upsilon\in \hat M$ is non-trivial,
then $\upsilon^*$ is non-trivial, and hence $\vecv$  is orthogonal to all $\scrU(\upsilon,s)_{\upsilon^*}$.
Therefore we deduce the following corollary:
\begin{cor}\label{sigmazero3} If $\upsilon\in \hat M$ is non-trivial, then for any $M$-invariant $\vecv\in \scrU(\upsilon,s)_{\tau_1}$,
$$T_{\tau_1}^{\tau_2}\vecv=0.$$ 
\end{cor}

\begin{lem}\label{Tadj} For any $K$-types $\tau_1, \tau_2$ of $\scrU(\upsilon,s)$, we have
\begin{equation*}
(T_{\tau_1}^{\tau_2})^*=T_{\tau_2}^{\tau_1},
\end{equation*}
where the adjoint is defined with respect to $\langle\cdot,\cdot\rangle_K$.
\end{lem}
\begin{proof}
For any $\vecu\in \scrU(\upsilon,s)_{\tau_1}$ and $\vecv\in \scrU(\upsilon,s)_{\tau_2}$, we have
\begin{align*}
\langle T_{\tau_1}^{\tau_2}\vecu,\vecv\rangle_K&=\int_M \vecu(m)\langle U^s(m)\vecchi_{\tau_2},\vecv\rangle_K\,dm
\\&=\int_M \langle \vecu, U^s(m)\vecchi_{\tau_1}\rangle_K\langle U^s(m)\vecchi_{\tau_2},\vecv\rangle_K\,dm
\\&=\left\langle\vecu,\int_M\langle U^s(m^{-1})\vecv,\vecchi_{\tau_2}\rangle_K\,\cdot\, U^s(m)\vecchi_{\tau_1}\,dm \right\rangle
\\&=\langle\vecu,T_{\tau_2}^{\tau_1}\vecv\rangle_K.
\end{align*}
\end{proof}
\begin{cor}\label{Tbd}  For any $K$-types $\tau_1, \tau_2$ of $\scrU(\upsilon,s)$,
\begin{equation*}
\|T_{\tau_1}^{\tau_2}\|_{K}\leq {\frac{\sqrt{\dim(\tau_1)\dim(\tau_2)}}{\dim(\upsilon)}}.
\end{equation*}
\end{cor}
\begin{proof}
For any unit vector $\vecv\in \scrU(\upsilon,s)_{\tau_1}\cap \scrU(\upsilon,s)_{\upsilon^*}$, we have
\begin{align*}
&\langle T_{\tau_1}^{\tau_2}\vecv,T_{\tau_1}^{\tau_2}\vecv\rangle_K
=\int_M \overline{\vecv(m)}\langle T_{\tau_1}^{\tau_2}\vecv,U^s(m)\vecchi_{\tau_2}\rangle_K\,dm
\\&=\int_M\int_M \overline{\vecv(m_1)}\vecv(m_2)\langle U^s(m_2)\vecchi_{\tau_2},U^s(m_1)\vecchi_{\tau_2}\rangle_K\,dm_2\,dm_1
\\&=\int_M\int_M \overline{\langle \vecv,U^s(m_1)\vecchi_{\tau_1}\rangle_K}\langle \vecv,U^s(m_2)\vecchi_{\tau_1}\rangle_K\vecchi_{\tau_2}(m_2^{-1}m_1)\,dm_2\,dm_1
\\&=\int_M\left(\int_M \overline{\langle \vecv,U^s(m_2m_3)\vecchi_{\tau_1}\rangle_K}\langle \vecv,U^s(m_2)\vecchi_{\tau_1}\rangle_K\,dm_2\right)\vecchi_{\tau_2}(m_3)\,dm_3.
\end{align*}
Since $\vecv$ and $\vecchi_{\tau_1}$ are both in $\scrU(\upsilon,s)_{\tau_1}\cap \scrU(\upsilon,s)_{\upsilon^*}$, the Schur orthogonality relations for $M$ give
\begin{align*}
\int_M& \overline{\langle \vecv,U^s(m_2m_3)\vecchi_{\tau_1}\rangle_K}\langle \vecv,U^s(m_2)\vecchi_{\tau_1}\rangle_K\,dm_2 
\\&=\frac{\|\vecv\|^2}{\dim(\upsilon)}\overline{\langle\vecchi_{\tau_1},U^s(m_3)\vecchi_{\tau_1}\rangle}=\frac{\overline{\vecchi_{\tau_1}(m_3)}}{\dim(\upsilon)}.
\end{align*}
We thus have 
\begin{align*}
\langle T_{\tau_1}^{\tau_2}\vecv,T_{\tau_1}^{\tau_2}\vecv\rangle_K&=\frac{1}{\dim(\upsilon)}\int_M\overline{\vecchi_{\tau_1}(m)}\vecchi_{\tau_2}(m)\,dm
\\ &\leq  \frac{1}{\dim(\upsilon)}\sqrt{\int_M|\vecchi_{\tau_1}(m)|^2\,dm\int_M|\vecchi_{\tau_2}(m)|^2\,dm}
\\ &= \frac{\dim(\tau_1)\dim(\tau_2)}{\dim(\upsilon)^2}
\end{align*}
by Lemma \ref{Mvint} (2).
Combining this estimate with Proposition  \ref{sigmazero}, the claim follows.
\end{proof}
\begin{cor}\label{Ttautau}
For any $K$-type $\tau$ of $\scrU(\upsilon,s)$,  we have
\begin{equation*}
T_{\tau}^{\tau}=\frac{\dim(\tau)}{\dim(\upsilon)} \mathsf{P}_{\upsilon^*}.
\end{equation*}
\end{cor}
\begin{proof}
Following the proof of Corollary \ref{Tbd}, we obtain
\begin{equation*}
\|T_{\tau}^{\tau}\|_K^2=\frac{1}{\dim(\upsilon)}\int_M|\vecchi_{\tau}(m)|^2\,dm=\frac{\dim(\tau)^2}{\dim(\upsilon)^2}.
\end{equation*}
So by Proposition \ref{sigmazero}, we have 
$$T_{\tau}^{\tau}=c\cdot\mathsf{P}_{\upsilon^*},$$ where $c$ is one of $\pm \frac{\dim(\tau)}{\dim(\upsilon)}$. Since $$c=\langle T_{\tau}^{\tau}\vecv,\vecv\rangle_K=\int_M|\vecv(m)|^2\,dm\geq 0$$
for any unit vector $\vecv\in\scrU(\upsilon,s)_{\tau}\cap\scrU(\upsilon,s)_{\upsilon^*}$,
we get $c=\frac{\dim(\tau)}{\dim(\upsilon)}$.
\end{proof}
\section{Asymptotic expansions of matrix coefficients}\label{matrixcoeffs}
We fix  a complementary series representation $\scrU(\upsilon,s)$ of $G$ for some $\upsilon\in \hat M$ and $s\in \scrI_{\upsilon}$.
The main goal of this section is to obtain  effective expansions of matrix coefficients for $\scrU(\upsilon,s)$. More precisely, we will write a matrix coefficient $\langle U^s(a_t) \vecv, \vecu\rangle_{\scrU(\upsilon,s)}$
as a main term  that decays like $e^{(s-d)t}$ as $t\rightarrow\infty$ and an error term that decays exponentially faster depending only on $2s-d>0$. To do this, we first work out the asymptotics of matrix coefficients with respect to $\langle \cdot,\cdot\rangle_K$ and then use explicit formulas for intertwining operators to convert our results into statements for $\langle \cdot,\cdot\rangle_{\scrU(\upsilon,s)}$.

\subsection{Matrix coefficients with respect to $\langle \cdot,\cdot,\rangle_K$.}
We start by proving a bound on how far elements of $K$ move vectors in representations of $K$. We let $\dist_K$ denote the bi-invariant Riemannian metric on $K$ induced from the negative of the Killing form on $\mathfrak{k}$, and
let $|\cdot|_K$ denote the corresponding norm on $\mathfrak{k}$.
\begin{lem}\label{KSob} Let $(\pi,V)$ be a finite-dimensional unitary representation of $K$ with invariant inner product $(\cdot,\cdot)_V$. Then
for all $\vecv\in V$ and $k\in K$, we have \begin{equation*}
\|\pi(k)\vecv-\vecv\|_{V} \ll \dist_K(k,e)\,\|\vecv\|_{\scrS^1(V)},
\end{equation*}
where the implied constant depends solely on the choice of basis defining $\|\cdot\|_{\scrS^m(V)}$.
\end{lem}
\begin{proof}
Since $\dist_K$ is bi-invariant, there exists $J\in \mathfrak{k}$ with $|J|_K=1$ such that the $k=\exp(\dist_K(k,e) J)$. This gives
\begin{align*}
\pi(k)\vecv-\vecv&=\int_0^{\dist_K(k,e)} \frac{d}{dt} \pi\big( \exp(tJ)\big)\vecv\,dt
\\&=\int_0^{\dist_K(k,e)}  \pi(\exp(tJ))d\pi\big( \Ad({\exp(-tJ)})J\big)\vecv\,dt.
\end{align*}
Since $\pi$ is unitary, $|J|_K=1$, and $K$ is compact, we get
\begin{equation*}
\big\| \pi(\exp(tJ))d\pi\big( \Ad({\exp(-tJ)})J\big)\vecv\big\|_V\leq\max_{\overset{X\in\mathfrak{k}}{|X|_K=1}}\|d\pi( X)\vecv\|_V\ll \|\vecv\|_{\scrS^1(V)}.
\end{equation*}
\end{proof}

\begin{defn} The Harish-Chandra $c$-function $C_+(s): \scrU(\upsilon,s)\to \scrU(\upsilon,s)$ is defined as follows:
\begin{equation}\label{ccc}
C_+(s)=\int_{\overline{N}} U^s\big(\kappa(\overline{n})^{-1}\big) e^{-s H(\overline{n})}\,d\overline{n}.
\end{equation}
\end{defn}
 Since $U^s$ is unitary on $\scrU(\upsilon,s)$ and $\int_{\overline{N}}  e^{-s H(\overline{n})}\,d\overline{n}<\infty$ for all $s>\frac{d}{2}$ (cf.\ \cite[Proposition 7.6]{Knapp1}), $C_+(s)$ is a well-defined bounded operator on $\scrU(\upsilon,s)$. Since $C_+(s)$ is defined using only the restriction of $U^s$ to $K$, it preserves the $K$-types of $\scrU(\upsilon,s)$.
 
 \begin{lem}\label{css}
 $C_+(s)$ preserves the $M$-types of $\scrU(\upsilon,s)$.
 \end{lem}
 \begin{proof} Note that $M$ normalizes $\overline{N}$, $d\overline{n}=d(m\overline{n}m^{-1})$ and $H(mgm^{-1})=H(g)$ for all $m\in M$, $\overline{n}\in\overline{N}$, and $g\in G$. Therefore for all $m\in M$,
\begin{align*}
C_+(s)U^s(m)&=\int_{\overline{N}} U^s\big(\kappa(\overline{n})^{-1}m\big) e^{-s H(\overline{n})}\,d\overline{n}
\\&=\int_{\overline{N}} U^s\big(m\kappa(m^{-1}\overline{n}m)^{-1}\big) e^{-s H(\overline{n})}\,d\overline{n}=U^s(m)C_+(s);
\end{align*}
hence the claim follows.
\end{proof}
For $d/2<s< d$, set
\begin{equation}\label{eta2}
\eta_s=\min\lbrace 2s-d, 1\rbrace >0.
\end{equation}
We remark that the following theorem was shown in \cite[Theorem 3.23]{MoOh} for {\it{some} $\eta_s$}, based on Harish-Chandra's expansion formula and the maximum modulus principle. We give a more direct argument, with the explicit $\eta_s$ given in \eqref{eta2}. Let $\scrS^1(K)$ denote the Sobolev norm $\scrS^1(L^2(K))$ defined in Section \ref{Sobnormdefsec}.
\begin{thm}\label{Kmatrixcoeffs} Let $\tau_1, \tau_2$ be $K$-types of $\scrU(\upsilon,s)$.
For all $\vecu\in \scrU(\upsilon,s)_{\tau_1}$ and $\vecv\in \scrU(\upsilon,s)_{\tau_2}$, we have for any $t\geq 0$,  
\begin{align*}
\langle U^s(a_t)\vecu,\vecv\rangle_{K}=e^{(s-d)t}&\langle T_{\tau_1}^{\tau_2}C_+(s)\vecu,\vecv\rangle_{K}
\\&\quad + O_s\left(e^{(s-d-\eta_s)t}\|T_{\tau_1}^{\tau_2}\|_K\|\vecu\|_{K} \|\vecv\|_{\scrS^1(K)}\right),
\end{align*}
where the implied constant is uniformly bounded over $s$ in compact subsets of $(\frac{d}{2},d)$.
\end{thm}
\begin{proof}
Applying Theorem \ref{Eisensteinthm}, we have 
\begin{equation*}
\langle U^s(a_t)\vecu,\vecv\rangle_{K}=\int_K e^{(s-d)H(a_tk)} \langle U^s\big(\kappa(a_tk)\big) T_{\tau_1}^{\tau_2}U^s(k^{-1})\vecu,\vecv\big\rangle_{K}\,dk.
\end{equation*}
Since the function $k\mapsto e^{(s-d)H(a_tk)} \langle U^s\big(\kappa(a_tk)\big) T_{\tau_1}^{\tau_2}U^s(k^{-1})\vecu,\vecv\big\rangle_{\scrU(\upsilon,s)}$ is right $M$-invariant, we may use the integration formula \cite[Consequence 3, p. 147]{Knapp1} to obtain
\begin{multline*}
\langle U^s(a_t)\vecu,\vecv\rangle_{K}\\=\int_{\overline{N}}e^{(s-d)H(a_t\kappa(\overline{n}))} \big\langle U^s\big(\kappa(a_t\kappa(\overline{n}))\big) T_{\tau_1}^{\tau_2}U^s(\kappa(\overline{n})^{-1})\vecu,\vecv\big\rangle_{K}e^{-d\,H(\overline{n})}\,d\overline{n}.
\end{multline*}
The identities
\begin{align*}
H\big(a_t\kappa(\overline{n})\big)&=H(a_t\overline{n}a_{-t})+H(a_t)-H(\overline{n})\quad\text{ and }\\
\kappa\big(a_t\kappa(\overline{n})\big)&=\kappa(a_t\overline{n}a_{-t})
\end{align*} 
then give that the previous integral is equal to
\begin{align*}
e^{(s-d)t}\int_{\overline{N}}\big\langle  T_{\tau_1}^{\tau_2}U^s(\kappa(\overline{n})^{-1})\vecu,U^s\big(\kappa(a_t\overline{n}a_{-t})^{-1}\big)\vecv\big\rangle_{K}e^{(s-d)H(a_t\overline{n}a_{-t})-s\,H(\overline{n})}\,d\overline{n}.
\end{align*}
We now use the identification of $\overline{N}$ with $\RR^d$ to again rewrite: 
\begin{equation}\label{Nint}
e^{(s-d)t}\int_{\RR^d}\big\langle  T_{\tau_1}^{\tau_2}U^s(\kappa(\overline{n}_{\vecx})^{-1})\vecu,U^s\big(\kappa(\overline{n}_{e^{-t}\vecx})^{-1}\big)\vecv\big\rangle_{K}e^{(s-d)H(\overline{n}_{e^{-t}\vecx})-s\,H(\overline{n}_{\vecx})}\,d\vecx;
\end{equation}
note that the integral is absolutely convergent due to $s>\frac{d}{2}$. Using the fact that  $e^{H(\overline{n}_{\vecx})}= 1+\|\vecx\|^2$ (cf. \cite[p.\ 486 and p.\ 564]{Knapp2}) gives
\begin{align*}
\langle& U^s(a_t)\vecu,\vecv\rangle_{K}=e^{(s-d)t}\langle T_{\tau_1}^{\tau_2}C_+(s)\vecu,\vecv\rangle_{K}+
\\&e^{(s-d)t}\int_{\RR^d} \big\langle  T_{\tau_1}^{\tau_2}U^s(\kappa(\overline{n}_{\vecx})^{-1})\vecu,(1+\|e^{-t}\vecx\|^2)^{s-d}U^s\big(\kappa(\overline{n}_{e^{-t}\vecx})^{-1}\big)\vecv-\vecv\big\rangle_{K}  \sfrac{d\vecx}{(1+\|\vecx\|^2)^{s}}.
\end{align*} 
Changing to spherical coordinates, we let $\vecx=(r, \vectheta)$, $r\geq 0$, $\vectheta\in \mathbb{S}^{d-1}$ and set
\begin{equation*}
\vecw(r,\vectheta)=(1+r^2)^{s-d}U^s\big(\kappa(\overline{n}_{(r,\vectheta)})^{-1}\big)\vecv-\vecv.
\end{equation*}
Using this, we deduce
\begin{align*}
\langle& U^s(a_t)\vecu,\vecv\rangle_{K}=e^{(s-d)t}\langle T_{\tau_1}^{\tau_2}C_+(s)\vecu,\vecv\rangle_{K}+
\\&O\left(e^{(s-d)t}\int_0^{\infty} \int_{\mathbb{S}^{d-1}} \big\langle  T_{\tau_1}^{\tau_2}U^s(\kappa(\overline{n}_{(r,\vectheta)})^{-1})\vecu,\vecw(e^{-t}r,\vectheta)\big\rangle_{K} \,dm(\vectheta)\,  \tfrac{r^{d-1}}{ (1+r^2)^{s}}\,dr\right),
\end{align*}
where $dm(\vectheta)$ denotes the spherical measure.

Since the map $\overline{n}\mapsto\kappa(\overline{n})$ is smooth, we get
\begin{equation*}
d_K\big(\kappa(\overline{n}_{(r,\vectheta)}),e\big)\ll r\qquad \text{for all } r>0, \vectheta\in\mathbb{S}^{d-1}.
\end{equation*}
Using $(1+r^2)^{s-d}= 1 -O_s(r)$ (with the implied constant depending continuously on $s$) and Lemma \ref{KSob}, we have
\begin{equation*}
\|\vecw(r,\theta)\|_K\ll_{s} \min\lbrace 1, r\rbrace\|\vecv\|_{\scrS^1(K)}.
\end{equation*}
This gives
\begin{align*}
&\langle U^s(a_t)\vecu,\vecv\rangle_{K}=e^{(s-d)t}\langle T_{\tau_1}^{\tau_2}C_+(s)\vecu,\vecv\rangle_{K}
\\&\qquad+ O_s\left(\|T_{\tau_1}^{\tau_2}\|_{K}\|\vecu\|_K \|\vecv\|_{\scrS^1(K)} e^{(s-d)t}\int_0^{\infty} \min\lbrace 1, e^{-t}r\rbrace (1+r^2)^{-s} r^{d-1}\,dr\right).
\end{align*}
The proof is completed by writing the integral $\int_0^{\infty}$ as $\int_0^1+\int_1^{e^t}+\int_{e^t}^{\infty}$ to obtain
\begin{align*}
&\int_0^{\infty} \min\lbrace 1, e^{-t}r\rbrace (1+r^2)^{-s} r^{d-1}\,dr
\\&\leq e^{-t}+e^{-t}\int_1^{e^t} r\,\cdot\,r^{-2s}\,\cdot\,r^{d-1}\,dr+\int_{e^t}^{\infty} r^{-2s}\,\cdot r^{d-1}\,dr
\\&\ll_s\, e^{-t} +e^{(d-2s)t}.
\end{align*} 
\end{proof}

\subsection{The invariant inner product on $\scrU(\upsilon,s)$}
The intertwining operator $\scrA(\upsilon,s)$ on $\scrU(\upsilon,s)$ is defined so that
\begin{equation*}
\langle \vecu,\vecv\rangle_{\scrU(\upsilon,s)}=\langle \vecu, \scrA(\upsilon,s)\vecv\rangle_K
\end{equation*}
for all $K$-finite vectors $\vecu,\vecv\in\scrU(\upsilon,s)$. The key intertwining property of $\scrA(\upsilon,s)$ reads (cf.\ \cite[Lemmas 22 and 23]{KnappStein})
\begin{equation}\label{intertwining}
\scrA(\upsilon,s)U^s(g)=U^{d-s}(g)\scrA(\upsilon,s)\qquad\text{for all $g\in G.$}
\end{equation}
In particular, $\scrA(\upsilon,s)$ commutes with $U^s(k)$ for all $k\in K$. Since each $K$-type occurs at most once in $\scrU(\upsilon,s)$, by Schur's lemma, $\scrA(\upsilon,s)$ acts as a scalar $a(\upsilon,s,\tau)$ on each $K$-type $\tau$ of $\scrU(\upsilon,s)$:
\begin{equation}\label{Adecomp}
\scrA(\upsilon,s)=\sum_{\tau\supset\upsilon} a(\upsilon,s,\tau)\mathsf{P}_{\tau}.
\end{equation}
The positive definiteness of the inner product $\langle\cdot,\cdot\rangle_{\scrU(\upsilon,s)}$ implies that  for all $\tau\in \hat K$
contained in $\scrU(\upsilon,s)$, we have
 $$a(\upsilon,s,\tau)>0.$$

Recalling the parameterization of $K$ and $M$ types given in Section \ref{SO(n)reps}, we now assume that $\scrU(\upsilon,s)$ has a non-trivial $M$-invariant vector. There is then a $K$-type $\sigma=(\sigma_1,\sigma_2,\ldots)$ of $\scrU(\upsilon,s)$ that contains the trivial representation of $M$. Thus $(\sigma_1,\sigma_2,\ldots)$ satisfies the interlacing relation with the trivial representation $(0,0,0,\ldots,0)$ of $M$: 

$$ \sigma_1 \geq 0 \geq \sigma_2\geq 0 \geq \sigma_3\geq\ldots. $$

From this, we conclude that $\sigma=(\sigma_1,0,0,\ldots,0)$. Now writing $\upsilon=(\upsilon_1,\upsilon_2,\ldots)$, the classification of $K$-types of $\scrU(\upsilon,s)$ ensures that $\upsilon$ is a subrepresentation of $\sigma$, and so $(\upsilon_1,\upsilon_2,\ldots)$ must satisfy the interlacing relation with $(\sigma_1,0,0,0,\ldots)$. We therefore see that $\upsilon=(\upsilon_1,0,0,0,\ldots)$. For notational convenience we will simply write $\upsilon=(\upsilon,0,0,0,\ldots)\in\ZZ^{\lfloor \frac{d}{2}\rfloor}$, where $\upsilon\geq 0$.
Note that if $d=1$ or $ 2$, then by the classification of the unitary dual of $\SO(2,1)$ and $\SO(3,1)$, $\upsilon=0$ \cite{Hirai}. Combining the fact that each $K$-type of $\scrU(\upsilon,s)$ must have $\upsilon$ as a subrepresentation with the interlacing relation gives that all $K$-types $\tau$ of $\scrU(\upsilon,s)$ may be parameterized as vectors $\tau=(t_1,t_2,0,\ldots,0)\in\ZZ^{\lfloor \frac{d+1}{2}\rfloor}$ with $t_1\geq \upsilon\geq t_2\geq 0$ if $d>3$, $t_1\geq \upsilon\geq |t_2|$ if $d=3$, $t_1\geq 0$ if $d=2$, and just $t_1\in\ZZ$ if $d=1$. We will thus write $\tau=(t_1,t_2)$ and $a(\upsilon,s,\tau)=a(\upsilon,s,t_1,t_2)$, with $\upsilon=t_2=0$ for $d=1,2$. Again using the classification of the unitary dual of $G$, observe that if $s\geq d-1$, then $\upsilon=t_2=0$. Finally, we let $\Omega_K\in Z(\fk_{\CC})$ denote the Casimir operator of $K$. 

The following lemma is the main technical result needed to establish the bounds on the quotients $\frac{a(\upsilon,s,\tau_1)}{a(\upsilon,s,\tau_2)}$ given in Propositions \ref{Abdd1} and \ref{Abdd2}. 

\begin{lem}\label{Gammas}
Assume that $\scrU(\upsilon,s)$ has a non-trivial $M$-invariant vector. Then for any $K$-types $\tau_1=(t_1,t_2)$ and $\tau_2=(t_3,t_4)$ of $\scrU(\upsilon,s)$, if $s<d-1$,
\begin{align*}
&\frac{a(\upsilon,s,\tau_2)}{a(\upsilon,s,\tau_1)}=\frac{\Gamma(d-s+t_3)}{\Gamma(s+t_3)}\,\cdot\, \frac{\Gamma(d-s+t_4-1)}{\Gamma(s+t_4-1)}\,\cdot\,\frac{\Gamma(s+t_1)}{\Gamma(d-s+t_1)}\,\cdot\, \frac{\Gamma(s+t_2-1)}{\Gamma(d-s+t_2-1)},
\end{align*}
and if $s\geq d-1$,
\begin{align*}
&\frac{a(\upsilon,s,\tau_2)}{a(\upsilon,s,\tau_1)}=\frac{\Gamma(d-s+t_3)}{\Gamma(s+t_3)}\,\cdot\, \frac{\Gamma(s+t_1)}{\Gamma(d-s+t_1)}.
\end{align*}
\end{lem}
\begin{proof}
The claimed formula follows from a recursion formula similar to that for the Harish-Chandra $c$-function given in \cite{Egg}. We start by letting $H\in\fa$ be such that $a_t=\exp(t H)$, and defining an inner product $\langle \cdot,\cdot\rangle_{\fg}$ on $\fg_{\CC}$ by 
$$\langle X,Y\rangle_{\fg}:=c\cdot\left( -B(X,\theta Y) \right),$$
where $B(\cdot,\cdot)$ denotes the Killing form on $\fg$, $\theta$ is the Cartan involution defining $K$, and $c\in\RR_{>0}$ is chosen so that $ \langle H,H\rangle_{\fg}=1$. Denote the $-1$-eigenspace of $\theta$ by $\mathfrak{p}$. Then for all $X\in\mathfrak{p}_{\CC}$, $\vecv\in C^{\infty}(K)$, and $k\in K$, we have 
\begin{align}\label{U(X)}
&[dU^s(X)\vecv](k)= (s-\sfrac{d}{2})\langle \Ad_{k^{-1}}X,H\rangle_{\fg} \vecv(k)
\\\notag&\quad\quad\quad+\sfrac{1}{2}\Big( d\rho(\Omega_K)\big\lbrace \langle \Ad_{k^{-1}}X,H\rangle_{\mathfrak{g}} \vecv(k) \big\rbrace -\langle \Ad_{k^{-1}}X,H\rangle_{\mathfrak{g}} [d\rho(\Omega_K)\vecv](k) \Big),
\end{align}
where $\rho$ denotes right-translation (cf.\ \cite[Lemma 1]{Thieleker} and \cite[Lemma 3.2]{Egg}). Let $\tau=(r_1,r_2)$ be an arbitrary $K$-type of $\scrU(\upsilon,s)$. By e.g. \cite[Lemma 2]{Thieleker} (cf.\ also \cite[Proposition 5.28]{Knapp2}), for any vector $\vecv\in\scrU(\upsilon,s)_{\tau}$, 
\begin{equation}\label{CasEig}
d\rho(\Omega_K)\vecv=dU^s(\Omega_K)\vecv=(r_1^2+r_2^2+(d-1)r_1+(d-3)r_2)\vecv.
\end{equation}
We now denote the orthogonal projection onto $\tau$ by $\mathsf{P}_{(r_1,r_2)}$. Combining the above expression for $d\rho(\Omega_K)\vecv$ with \eqref{U(X)} gives
\begin{align}
\notag&[\mathsf{P}_{(r_1+1,r_2)}dU^s(X)\vecv](k)
\\\notag&= \left( s-\sfrac{d}{2} +\sfrac{(r_1+1)^2+r_2^2+(d-1)(r_1+1)+(d-3)r_2}{2}-\sfrac{r_1^2+r_2^2+(d-1)r_1+(d-3)r_2}{2} \right)
\\\notag &\qquad\qquad\qquad\qquad\qquad\qquad\qquad\times \mathsf{P}_{(r_1+1,r_2)}\left\lbrace \langle \Ad_{k^{-1}}X,H\rangle_{\fg} \vecv(k) \right\rbrace
\\&= \left( s +r_1\right)\mathsf{P}_{(r_1+1,r_2)}\left\lbrace \langle \Ad_{k^{-1}}X,H\rangle_{\fg} \vecv(k) \right\rbrace, \label{scalar1}
\end{align}
and similarly
\begin{align}
[\mathsf{P}_{(r_1,r_2+1)}dU^s(X)\vecv](k)
= \left( s +r_2-1\right)\mathsf{P}_{(r_1,r_2+1)}\left\lbrace \langle \Ad_{k^{-1}}X,H\rangle_{\fg} \vecv(k) \right\rbrace. \label{scalar2}
\end{align}
From \eqref{intertwining}, we obtain
\begin{equation*}
[\mathsf{P}_{(r_1,r_2)}dU^s(X)\scrA(\upsilon,s)\vecv](k)=[\mathsf{P}_{(r_1,r_2)}\scrA(\upsilon,s)dU^{d-s}(X)\vecv](k).
\end{equation*}
Combining this with \eqref{Adecomp} and \eqref{scalar1} and \eqref{scalar2}, respectively, gives
\begin{align*}
&a(\upsilon,s,r_1,r_2) (s+r_1)\mathsf{P}_{(r_1+1,r_2)}\left\lbrace \langle \Ad_{k^{-1}}X,H\rangle_{\fg} \vecv(k) \right\rbrace
\\&\qquad=a(\upsilon,s,r_1+1,r_2) (d-s+r_1)\mathsf{P}_{(r_1+1,r_2)}\left\lbrace \langle \Ad_{k^{-1}}X,H\rangle_{\fg} \vecv(k) \right\rbrace,
\end{align*}
and
\begin{align*}
&a(\upsilon,s,r_1,r_2) (s+r_2-1)\mathsf{P}_{(r_1,r_2+1)}\left\lbrace \langle \Ad_{k^{-1}}X,H\rangle_{\fg} \vecv(k) \right\rbrace
\\&\qquad=a(\upsilon,s,r_1,r_2+1) (d-s+r_2-1)\mathsf{P}_{(r_1,r_2+1)}\left\lbrace \langle \Ad_{k^{-1}}X,H\rangle_{\fg} \vecv(k) \right\rbrace.
\end{align*}

The decomposition of $(\Ad,\mathfrak{p}_{\CC})\otimes (U^s,\scrU(\upsilon,s)_{\tau})$ into irreducible representations of $K$ ensures the existence of $X,Y\in\mathfrak{p}_{\CC}$ and $\vecv_1,\vecv_2\in\scrU(\upsilon,s)_{\tau}$ such that $\mathsf{P}_{(r_1+1,r_2)}\left\lbrace \langle \Ad_{k^{-1}}X,H\rangle_{\fg} \vecv_1(k)\right\rbrace \not\equiv 0$, and $\mathsf{P}_{(r_1,r_2+1)}\left\lbrace \langle \Ad_{k^{-1}}Y,H\rangle_{\fg} \vecv_2(k)\right\rbrace \not\equiv 0$ if $(r_1+1,r_2)$, and $(r_1,r_2+1)$ are $K$-types of $\scrU(\upsilon,s)$, cf., e.g., \cite[Lemma 3]{Thieleker} and \cite[Theorem 3.4.12]{Wallach1}. This gives
\begin{equation*}
(s+r_1)\cdot a(\upsilon,s,r_1+1,r_2)=(d-s+r_1)\cdot a(\upsilon,s,r_1,r_2)
\end{equation*}
and
\begin{equation*}
(s+r_2-1)\cdot a(\upsilon,s,r_1,r_2+1)=(d-s+r_2-1)\cdot a(\upsilon,s,r_1,r_2)
\end{equation*}
(cf.\ \cite[(6.1) and (6.7)]{Egg}). Note that if one of the four factors that are multiplied with $a(\upsilon,s,\ldots)$'s is zero, then the representation given by that choice of $\upsilon$ and $s$ is not in the unitary dual of $G$. In particular, if $s\geq d-1$, then $\scrU(\upsilon,s)$ is spherical, hence $\upsilon=r_2=0$. These two recursion formulas imply that if $\upsilon\neq 0$,
\begin{align*}
a(\upsilon,s,r_1,r_2)=& \tfrac{\Gamma(d-s+r_1)}{\Gamma(s+r_1)}\,\cdot\, \tfrac{\Gamma(d-s+r_2-1)}{\Gamma(s+r_2-1)}\,\cdot\, \tfrac{\Gamma(s+\upsilon)}{\Gamma(d-s+\upsilon)}\cdot\tfrac{\Gamma(s+\upsilon-1)}{\Gamma(d-s+\upsilon-1)}\,\cdot\, a(\upsilon,s,\upsilon,\upsilon),
\end{align*}
and if $\upsilon=0$,
\begin{align*}
a(0,s,r_1,0)=& \tfrac{\Gamma(d-s+r_1)}{\Gamma(s+r_1)}\,\cdot\, \tfrac{\Gamma(s)}{\Gamma(d-s)}\,\cdot\, a(0,s,0,0).
\end{align*}
The formulas claimed in the proposition then follow.
\end{proof}

By Schur's lemma, $\Omega_K$ acts on any realization of a $K$-type $\tau$ by the same scalar, which we denote $\tau(\Omega_K)$.

\begin{prop}\label{Abdd1} Assume that $\scrU(\upsilon,s)$ has a non-trivial $M$-invariant vector. Then for any $K$-types $\tau_1,\tau_2$ of $\scrU(\upsilon,s)$,
\begin{equation*}
\frac{a(\upsilon,s,\tau_2)}{a(\upsilon,s,\tau_1)} \ll_s \big(1+\tau_1(\Omega_K)^{d}\big)\big(1+\tau_2(\Omega_K)^{d}\big),
\end{equation*}
and the implied constant is uniformly bounded over $s$ in compact subsets of $\scrI_{\upsilon}$.
\end{prop}
\begin{proof}
The key result needed in the proof is a consequence of \cite[Theorem 1]{KeckicVasic}: for all $t\geq 0$, 
\begin{equation}\label{KV}
\tfrac{\Gamma(s+t)}{\Gamma(d-s+t)}\asymp_s 1+t^{2s-d},
\end{equation}
where the implied constant is uniformly bounded over $s$ in compact subsets of $(\frac{d}{2},d)$. We now write $\tau_1=(t_1,t_2)$, $\tau_2=(t_3,t_4)$.

Assuming first that $d>3$ and $\upsilon>0$, we then have $t_i\geq 0$ for all $i$. Since $\upsilon>0$,  $s<d-1$, and so rewriting Lemma \ref{Gammas} gives
\begin{align*}
&\tfrac{a(\upsilon,s,\tau_2)}{a(\upsilon,s,\tau_1)}=\tfrac{\Gamma(d-s+t_3)}{\Gamma(s+t_3)}\,\cdot\, \tfrac{\Gamma(d-s+t_4)}{\Gamma(s+t_4)}\,
\cdot\,\tfrac{\Gamma(s+t_1)}{\Gamma(d-s+t_1)}\,\cdot\, \tfrac{\Gamma(s+t_2)}{\Gamma(d-s+t_2)}
\,\cdot\, \tfrac{s+t_4-1}{d-s+t_4-1}\,\cdot\,\tfrac{d-s+t_2-1}{s+t_2-1},
\end{align*}
and so (after recalling the formula \eqref{CasEig} for $\tau_1(\Omega_K)$ and $\tau_2(\Omega_K)$),
\begin{align*}
\tfrac{a(\upsilon,s,\tau_2)}{a(\upsilon,s,\tau_1)} \asymp_s &(1+t_1^{2s-d})(1+t_2^{2s-d})(1+t_3^{2s-d})(1+t_3^{2s-d})
\\&\ll  \big(1+\tau_1(\Omega_K)^{d}\big)\big(1+\tau_2(\Omega_K)^{d}\big),
\end{align*}
with the implied constant uniformly bounded over $(\frac{d}{2},d-1)=\scrI_{\upsilon}$.

In the case $d\geq 2$ and $\upsilon=0$ (and so $t_2=t_4=0$), Lemma \ref{Gammas} gives
\begin{align*}
&\tfrac{a(\upsilon,s,\tau_2)}{a(\upsilon,s,\tau_1)}=\tfrac{\Gamma(d-s+t_3)}{\Gamma(s+t_3)}\,\cdot\, \tfrac{\Gamma(s+t_1)}{\Gamma(d-s+t_1)},
\end{align*}
so by a direct application of \eqref{CasEig} and \eqref{KV},
\begin{equation}\label{upsilon0bdd}
\tfrac{a(\upsilon,s,\tau_2)}{a(\upsilon,s,\tau_1)}\ll_s (1+ t_1^{2s-d})(1+ t_3^{2s-d}) \ll \big(1+\tau_1(\Omega_K)^{\frac{d}{2}}\big)\big(1+\tau_2(\Omega_K)^{\frac{d}{2}}\big),
\end{equation}
with the implied constant uniformly bounded over $s$ in compact subsets of $(\frac{d}{2},d)$.

For the remaining cases, $d=1$ and $d=3$ with $\upsilon>0$, negative values of the $t_i$ can appear in the formulas given in Lemma \ref{Gammas}. It thus remains to bound quotients of the form $ \frac{\Gamma(s+t)}{\Gamma(d-s+t)}$, where $t$ is a negative integer and $s\in (\frac{1}{2},1)$ if $d=1$, and $s\in (\frac{3}{2},2)$ if $d=3$. By the reflection formula, in both cases,
\begin{equation}\label{reflect}
\tfrac{\Gamma(s+t)}{\Gamma(d-s+t)}=\tfrac{\Gamma(s-|t|)}{\Gamma(d-s-|t|)}=\tfrac{\Gamma(|t|+1+s-d)}{\Gamma(|t|+1-s)}.
\end{equation}
If $d=1$, we then have 
\begin{equation*}
\tfrac{\Gamma(s+t)}{\Gamma(1-s+t)}=\tfrac{\Gamma(|t|+s)}{\Gamma(|t|+1-s)},
\end{equation*}
and so \eqref{KV} gives
\begin{align}\label{d1bd}
\tfrac{a(\upsilon,s,\tau_2)}{a(\upsilon,s,\tau_1)}&=\tfrac{\Gamma(1-s+|t_3|)}{\Gamma(s+|t_3|)}\,\cdot\, \tfrac{\Gamma(s+|t_1|)}{\Gamma(1-s+|t_1|)} \\\notag&\ll_s(1+ |t_1|^{2s-1})(1+ |t_2|^{2s-1}) \ll \big(1+\tau_1(\Omega_K)^{\frac{1}{2}}\big)\big(1+\tau_2(\Omega_K)^{\frac{1}{2}}\big).
\end{align}
If $d=3$, using \eqref{reflect}, we have
\begin{equation*}
\tfrac{\Gamma(s+t)}{\Gamma(3-s+t)}=\tfrac{\Gamma(|t|+s-2)}{\Gamma(|t|+2-s)}=\tfrac{(|t|+1-s)(|t|+2-s)}{(s+|t|-1)(s+|t|-2)}\,\cdot\,\tfrac{\Gamma(s+|t|)}{\Gamma(3-s+|t|)}.
\end{equation*}
Now using $\frac{(|t|+1-s)(|t|+2-s)}{(s+|t|-1)(s+|t|-2)}\asymp_s 1$, where the implied constants are uniformly bounded over $s$ in compact subsets of $(\frac{3}{2},2)$, together with \eqref{KV} as previously completes the proof.
\end{proof}

For $K$-types that contain $M$-invariant vectors, we have the following strengthening of the bound in Proposition \ref{Abdd1}: 

\begin{prop}\label{Abdd2} Let $\tau_1,\tau_2$ be $K$-types  of $\scrU(\upsilon,s)$ containing non-trivial $M$-invariant vectors. Then
\begin{equation*}
\frac{a(\upsilon,s,\tau_2)}{a(\upsilon,s,\tau_1)} \ll_s \big(1+\tau_1(\Omega_K)^{\frac{d}{2}}\big)\big(1+\tau_2(\Omega_K)^{\frac{d}{2}}\big),
\end{equation*}
and the implied constant is uniformly bounded over $s$ in compact subsets of $(\frac{d}{2},d)$.
\end{prop}
\begin{proof}
For $d\geq 3$, we observe that if a $K$-type $\tau=(r_1,r_2)$ contains an $M$-invariant vector, the interlacing relation implies $r_2=0$, hence for $t_1,t_2$ as in the statement of the lemma, we have $\tau_1=(t_1,0)$ and $\tau_2=(t_2,0)$. Lemma \ref{Gammas} in this case reads
\begin{align*}
&\tfrac{a(\upsilon,s,\tau_2)}{a(\upsilon,s,\tau_1)}=\tfrac{\Gamma(d-s+t_2)}{\Gamma(s+t_2)}\,\cdot\, \tfrac{\Gamma(s+t_1)}{\Gamma(d-s+t_1)}.
\end{align*}
The lemma then follows from \eqref{upsilon0bdd} for $d\geq 2$ and \eqref{d1bd} for $d=1$.
\end{proof}
\begin{Rmk}
The main point of Proposition \ref{Abdd2} is that the implied constant remains bounded even as $s$ approaches $d-1$ in the case $\upsilon>0$. This allows us to use the bound uniformly over complementary series representations appearing in the direct integral decomposition of $L^2(\GaG)$.
\end{Rmk}

The bounds from Propositions \ref{Abdd1} and \ref{Abdd2} allow us to restate Theorem \ref{Kmatrixcoeffs} in terms of $\langle\cdot,\cdot\rangle_{\scrU(\upsilon,s)}$. This can then be applied to the matrix coefficients of irreducible unitary representations weakly contained in $L^2(\GaG)$. Retaining the notation of Theorem \ref{Kmatrixcoeffs}, we have: 
\begin{prop}\label{Umatrixcoeffs}
There exists $m\in\NN$ such that for any $\scrU(\upsilon,s)$ {with a non-trivial $M$-invariant vector},
for all $\vecu\in \scrU(\upsilon,s)_{\tau_1},\vecv\in \scrU(\upsilon,s)_{\tau_2}$, and $t\geq 0$, 
\begin{align*}
\langle U^s(a_t)\vecu,\vecv\rangle_{\scrU(\upsilon,s)}=&e^{(s-d)t}\langle T_{\tau_1}^{\tau_2}C_+(s)\vecu,\vecv\rangle_{\scrU(\upsilon,s)}
\\&
\qquad\qquad+O_s\big(e^{(s-d-\eta_s)t}\|\vecu\|_{\scrS^m(\upsilon,s)}\|\vecv\|_{\scrS^m(\upsilon,s)} \big),
\end{align*}
where the implied constant is uniformly bounded over $s$ in compact subsets of $\scrI_{\upsilon}$. Furthermore, if the $K$-types $\tau_1$ and $\tau_2$ both contain non-trivial $M$-invariant vectors, the the implied constant is uniformly bounded over $s$ in compact subsets of $(\frac{d}{2},d)$.
\end{prop}
\begin{proof}
Since $\vecv\in \scrU(\upsilon,s)_{\tau_2}$, using the expression for $\langle\cdot,\cdot\rangle_{\scrU(\upsilon,s)}|_{\scrU(\upsilon,s)_{\tau_2}}$, Theorem \ref{Kmatrixcoeffs} gives
\begin{align*}
&\langle U^s(a_t)\vecu,\vecv\rangle_{\scrU(\upsilon,s)}=a(\upsilon,s,\tau_2)\langle U^s(a_t)\vecu,\vecv\rangle_{K}\\
&=\begin{multlined}[t]
e^{(s-d)t}a(\upsilon,s,\tau_2)\langle T_{\tau_1}^{\tau_2}C_+(s)\vecu,\vecv\rangle_{K}
 \\+  O_s\left(e^{(s-d-\eta_s)t} a(\upsilon,s,\tau_2)\|T_{\tau_1}^{\tau_2}\|_K\|\vecu\|_{K} \|\vecv\|_{\scrS^1(K)}\right)
\end{multlined} \\&=\begin{multlined}[t] e^{(s-d)t}\langle T_{\tau_1}^{\tau_2}C_+(s)\vecu,\vecv\rangle_{\scrU(\upsilon,s)}
\\ +O_s\left(e^{(s-d-\eta_s)t}\sqrt{\sfrac{a(\upsilon,s,\tau_2)}{a(\upsilon,s,\tau_1)}}\|T_{\tau_1}^{\tau_2}\|_K\|\vecu\|_{\scrU(\upsilon,s)}\|\vecv\|_{\scrS^1(\upsilon,s)}\right).\end{multlined}
\end{align*}
By Proposition \ref{Abdd1}, or Proposition \ref{Abdd2} if $\tau_1$ and $\tau_2$ both have $M$-invariant vectors,
\begin{align*}
\sqrt{\sfrac{a(\upsilon,s,\tau_2)}{a(\upsilon,s,\tau_1)}}&\|\vecu\|_{\scrU(\upsilon,s)}\|\vecv\|_{\scrS^1(\upsilon,s)}\\\ll_s& \sqrt{(1+\tau_1(\Omega_K)^{d})(1+\tau_2(\Omega_K)^{d})}\|\vecu\|_{\scrU(\upsilon,s)}\|\vecv\|_{\scrS^1(\upsilon,s)}
\\\ll& \|(1+dU^s(\Omega_K^{\lceil d/2 \rceil}))\vecu\|_{\scrU(\upsilon,s)}\|(1+dU^s(\Omega_K^{\lceil d/2 \rceil}))\vecv\|_{\scrS^1(\upsilon,s)}
\\&\ll  \|\vecu\|_{\scrS^{d+1}(\upsilon,s)}\|\vecv\|_{\scrS^{d+2}(\upsilon,s)},
\end{align*}
with the implied constant uniformly bounded over $s$ in compact subsets of $\scrI_{\upsilon}$, or $(\frac{d}{2},d)$, respectively. Now Corollary \ref{Tbd} and Lemma \ref{Fourier} give
\begin{align*}
\|T_{\tau_1}^{\tau_2}\|_K\|\vecu\|_{\scrS^{d+1}(\upsilon,s)}\|\vecv\|_{\scrS^{d+2}(\upsilon,s)}\leq &\sqrt{\dim(\tau_1)\dim(\tau_2)}\|\vecu\|_{\scrS^{d+1}(\upsilon,s)}\|\vecv\|_{\scrS^{d+2}(\upsilon,s)}\\&\ll \|\vecu\|_{\scrS^{m}(\upsilon,s)}\|\vecv\|_{\scrS^{m}(\upsilon,s)}
\end{align*}
for some $m\in \NN$, completing the proof.
\end{proof}
{\begin{Rmk}
Both Proposition \ref{Abdd1} and Proposition \ref{Umatrixcoeffs} are expected to hold for all complementary series $\scrU(\upsilon,s)$. We have proved them only for those representations with $M$-invariant vectors since in this case slight simplifications occur in the proof of Lemma \ref{Gammas}. Note that if Proposition \ref{Abdd1} were proved for all complementary series representations, Proposition \ref{Umatrixcoeffs} would also follow automatically for all complementary series.\end{Rmk}}
\begin{thm}\label{trivprop}
There exists $m\in\NN$ such that for any complementary series representation $\scrU(\upsilon,s)$ containing a non-trivial $M$-invariant vector, 
for all $\vecu,\vecv\in \scrS^m(\upsilon,s)$ and $t\geq 0$,
\begin{align*}
\langle U^s(a_t)\vecu,\vecv\rangle_{\scrU(\upsilon,s)} =&e^{(s-d)t}\left(\sum_{\tau_1,\tau_2\in{\hat K}}\langle T_{\tau_1}^{\tau_2}C_+(s)\mathsf{P}_{\tau_1}\vecu,\mathsf{P}_{\tau_2}\vecv\rangle_{\scrU(\upsilon,s)}\right)
\\&
\qquad\qquad+O_s\big(e^{(s-d-\eta_s)t}\|\vecu\|_{\scrS^m(\upsilon,s)}\|\vecv\|_{\scrS^m(\upsilon,s)} \big),
\end{align*}
 and the sum
\begin{equation}\label{maintermsum}
\sum_{\tau_1,\tau_2\in{\hat K}}\langle T_{\tau_1}^{\tau_2}C_+(s)\mathsf{P}_{\tau_1}\vecu,\mathsf{P}_{\tau_2}\vecv\rangle_{\scrU(\upsilon,s)}
\end{equation}
converges absolutely.
\end{thm}
\begin{proof}
Since smooth vectors are dense in $\scrS^m(\upsilon,s)$ for all $m\in\NN$, and both sides of the inequality are continuous with respect to $\|\cdot\|_{\scrS^m(\upsilon,s)}$, we start by assuming that $\vecu$ and $\vecv$ are smooth, and decompose them according to the $K$-types of $\scrU(\upsilon,s)$:
\begin{equation*}
\vecu=\sum_{\tau_1\subset \scrU(\upsilon,s)} \vecu_{\tau_1},\qquad \vecv=\sum_{\tau_2\subset \scrU(\upsilon,s)} \vecv_{\tau_2}, 
\end{equation*}
where $\vecu_{\tau_1}=\mathsf{P}_{\tau_1}\vecu$ and $\vecu_{\tau_2}=\mathsf{P}_{\tau_2}\vecv$. By \cite[Theorem 4.4.2.1]{Warner1},
\begin{equation*}
\langle U^s(a_t)\vecu,\vecv\rangle_{\scrU(\upsilon,s)}=\sum_{\tau_1,\tau_2} \langle U^s(a_t)\vecu_{\tau_1},\vecv_{\tau_2}\rangle_{\scrU(\upsilon,s)},
\end{equation*}
with the sum converging absolutely. Applying Proposition \ref{Umatrixcoeffs} gives
\begin{align}\label{MATPROJ}
\langle U^s(a_t)\vecu_{\tau_1},\vecv_{\tau_2}\rangle_{\scrU(\upsilon,s)}=e^{(s-d)t}&\langle T_{\tau_1}^{\tau_2}C_+(s)\vecu_{\tau_1},\vecv_{\tau_2}\rangle_{\scrU(\upsilon,s)}
\\\notag&+O_s\big(e^{(s-d-\eta_s)t}\|\vecu_{\tau_1}\|_{\scrS^{m'}(\upsilon,s)}\|\vecv_{\tau_2}\|_{\scrS^{m'}(\upsilon,s)} \big)
\end{align}
for some $m'\in\NN$. Hence 
\begin{multline*}
\langle U^s(a_t)\vecu,\vecv\rangle_{\scrU(\upsilon,s)} =e^{(s-d)t}\left(\sum_{\tau_1,\tau_2}\langle T_{\tau_1}^{\tau_2}C_+(s)\mathsf{P}_{\tau_1}\vecu,\mathsf{P}_{\tau_2}\vecv\rangle_{\scrU(\upsilon,s)}\right)
\\+O_s\left( e^{(s-d-\eta_s)t}\left(\sum_{\tau}\|\vecu_{\tau}\|_{\scrS^{m'}(\upsilon,s)} \right) \left(\sum_{\tau}\|\vecv_{\tau}\|_{\scrS^{m'}(\upsilon,s)} \right)\right).
\end{multline*}
By Lemma \ref{Fourier}, there exists $m\ge m'$ (depending only on $K$) large enough such that
$$\sum_{\tau\in \scrU(\upsilon, s)}\|\vecu_{\tau}\|_{\scrS^{m'}(\upsilon,s)} \ll   \|\vecu\|_{\scrS^{m}(\upsilon,s)}
\quad\text{ and }\quad \sum_{\tau\in \scrU(\upsilon, s)}\|\vecv_{\tau}\|_{\scrS^{m'}(\upsilon,s)} \ll  \|\vecv\|_{\scrS^{m}(\upsilon,s)}$$
where the implied constants depend only on $K$.

It remains to prove that the sum  $\sum_{\tau_1,\tau_2}\langle T_{\tau_1}^{\tau_2}C_+(s)\mathsf{P}_{\tau_1}\vecu,\mathsf{P}_{\tau_2}\vecv\rangle_{\scrU(\upsilon,s)}$  converges absolutely.  Looking at an individual summand, we have
\begin{align*}
|\langle T_{\tau_1}^{\tau_2}C_+(s)\vecu_{\tau_1},\vecv_{\tau_2}\rangle_{\scrU(\upsilon,s)}| &=a(\upsilon,s,\tau_2)|\langle T_{\tau_1}^{\tau_2}C_+(s)\vecu_{\tau_1},\vecv_{\tau_2}\rangle_{K}|
\\&\leq a(\upsilon,s,\tau_2)\|T_{\tau_1}^{\tau_2}\|_{K} \cdot \|C_+(s)\vecu_{\tau_1}\|_K \cdot \|\vecv_{\tau_2}\|_K
\end{align*}
Since $U^s|_K$ is unitary on $L^2(K)$, from the definition of $C_+(s)$ (see Theorem \ref{Kmatrixcoeffs}), 
we get $$\|C_+(s)\vecu_{\tau_1}\|_K\ll_s \|\vecu_{\tau_1}\|_K,$$ giving
\begin{align*}
|\langle T_{\tau_1}^{\tau_2}C_+(s)\vecu_{\tau_1},\vecv_{\tau_2}\rangle_{\scrU(\upsilon,s)}|&\ll_s a(\upsilon,s,\tau_2)\|T_{\tau_1}^{\tau_2}\|_{K} \cdot \|\vecu_{\tau_1}\|_K \cdot \|\vecv_{\tau_2}\|_K
\\&\leq \sqrt{\sfrac{a(\upsilon,s,\tau_2)}{a(\upsilon,s,\tau_1)}}\|T_{\tau_1}^{\tau_2}\|_{K} \cdot \|\vecu_{\tau_1}\|_{\scrU(\upsilon,s)} \cdot \|\vecv_{\tau_2}\|_{\scrU(\upsilon,s)}.
\end{align*}
This expression is now bounded using Proposition \ref{Abdd1}, Corollary \ref{Tbd}, and Lemma \ref{Fourier} as in the proof of Proposition \ref{Umatrixcoeffs}. Lemma \ref{Fourier} then gives the desired convergence of the sum.
\end{proof}
\begin{thm} \label{Minvmcoeffsbdd} There exists $m\in \NN$ such that for any non-spherical complementary series representation $\scrU(\upsilon,s)$, for all $M$-invariant vectors $\vecu$, $\vecv \in\scrS^m(\upsilon,s)$, and $t\geq 0$, we have
\begin{equation*}
 |\langle U^s(a_t)\vecu,\vecv\rangle_{\scrU(\upsilon,s)}|\ll_s e^{(s-d-\eta_s)t}\|\vecu\|_{\scrS^m(\upsilon,s)}\|\vecv\|_{\scrS^m(\upsilon,s)},
\end{equation*}
where the implied constant is uniformly bounded over $s$ in compact subsets of $(\frac{d}{2},d)$.
\end{thm}
\begin{proof}
Note that since $\vecu$, $\vecv$ are both $M$-invariant and $M\subset K$, $\mathsf{P}_{\tau}\vecu$ and $\mathsf{P}_{\tau}\vecv$ are as well for any $K$-type $\tau$ of $\scrU(\upsilon,s)$. Proposition \ref{Umatrixcoeffs} then gives that the implied constant in Theorem \ref{trivprop} is uniformly bounded over $s$ in compact subsets of $(\frac{d}{2},d)$. Thus, in order to prove the theorem, it suffices to show that for any $K$-types $\tau_1, \tau_2$ of $\scrU(\upsilon,s)$ and an arbitrary $M$-invariant vector $\vecw$ of $\scrU(\upsilon,s)$,
\begin{equation}\label{zzz} T_{\tau_1}^{\tau_2}C_+(s)\mathsf{P}_{\tau_1}\vecw=0.\end{equation}
Since $C_+(s)$ preserves the $M$-types of each $K$-type by Lemma \ref{css} and $\vecw$ is $M$-invariant (hence $\mathsf{P}_{\tau_1}\vecw$ is as well), $C_+(s)\mathsf{P}_{\tau_1}\vecw$ is $M$-invariant. Therefore \eqref{zzz} follows  from  Corollary \ref{sigmazero3}.
\end{proof}
\section{Leading term for Sobolev functions}\label{Sobmixingsec}
In this section, we will extend Roblin's result on mixing of the geodesic flow for continuous functions with compact support to general functions in a Sobolev space of sufficiently high order. 

Recall the following theorem of Roblin \cite[Theorem 3.4]{Roblin}:
\begin{thm}\label{roblin}
For all $M$-invariant $f_1,f_2\in C_c(\GaG)$,
\begin{equation*}
\lim_{t\rightarrow +\infty} e^{(d-\delta)t}\langle \rho(a_t)f_1,f_2\rangle=m^{\mathrm{BR}}(f_1)\,m^{\mathrm{BR_*}}(f_2).
\end{equation*}
\end{thm}

We note that $m^{\BR}(f)$ may be infinite for a general function $f\in L^2(\Gamma\ba G)$, and hence
Theorem \ref{roblin} does not generalize to arbitrary $L^2$-functions.

In order to extend this theorem to Sobolev functions (which are not necessarily compactly supported),  
we first prove the following bound on matrix coefficients of $L^2(\GaG)$:
\begin{lem}\label{mixingbd}
There exists $m\in \NN$ such that for all $f_1,f_2\in\scrS^m(\GaG)$ and $t\geq 0$,
\begin{equation*}
|\langle \rho(a_t)f_1,f_2\rangle|\ll_{\Gamma} e^{(\delta-d)t}\scrS^m(f_1)\scrS^m(f_2).
\end{equation*}
\end{lem}
\begin{proof} 
We start by assuming that $f_1$ and $f_2$ are $\rho(K)$-invariant.
Using the notation of Section \ref{repdefsec}, the decomposition of the functions according to $L^2(\GaG)_{\mathrm{sph}}=\scrB_{\delta}\oplus \scrW$ reads
\begin{equation*}
f_i = \langle f_i,\phi_0\rangle\phi_0+f_i',\qquad i=1,2,
\end{equation*}
where $f_i'\in \scrW$ are $K$-invariant, and $\phi_0\in \scrB_{\delta} $ is the base-eigenfunction in 
$L^2(\Gamma\ba \bH^{d+1})=L^2(\Gamma\ba G)^K$ of unit norm.

Since $\scrB_\delta$ is a complementary series representation
$\scrU (1, \delta)$ (cf. Sections \ref{repdefsec} and \ref{compserdefsec}), it follows from Theorem \ref{trivprop} that for all $t\ge 0$,
\begin{equation*}
|\langle\rho(a_t)\phi_0,\phi_0\rangle|\ll e^{(\delta-d)t}.
\end{equation*}

As a consequence of Theorem \ref{spectralgap}, $f_1'$ and $f_2'$ are orthogonal to any complementary series representation $\scrU(\upsilon,s)$ with $s_1<s\le \delta$. 
  It follows that for any $\epsilon>0$, for any $K$-invariant $f_1,f_2\in L^2(\GaG)$ and $t>0$,
\begin{equation*}
|\langle \rho(a_t)f_1',f_2'\rangle|\ll_{\epsilon} e^{(s_1-d+\epsilon)t}\|f_1'\|\|f_2'\|
\end{equation*}
(see \cite[Theorem 2.1, 2]{Shalom}, \cite[Proposition 5.3]{KontOh}).

Therefore, choosing $0<\epsilon< \delta-s_1$ gives
\begin{equation*}
|\langle \rho(a_t) f_1,f_2\rangle| \ll e^{(\delta-d)t}\big( |\langle f_1,\phi_0\rangle||\langle f_2,\phi_0\rangle|+\|f_1'\|\|f_2'\|\big)\leq e^{(\delta-d)t} \|f_1\|\|f_2\|.
\end{equation*}

In view of \cite[Proposition 2.5]{Shalom}, this bound extends for  all {\it $K$-finite} functions $f_1, f_2$ in $L^2(\GaG)$:
\begin{equation*}
| \langle \rho(a_t)f_1,f_2\rangle|\ll e^{(\delta-d)t}\|f_1\|\|f_2\|\sqrt{\dim(\rho(K)f_1)\,\dim(\rho(K)f_2)}.
\end{equation*}
To pass to Sobolev functions, we  first observe that
 $$\dim(\rho(K)\mathsf{P}_{\tau}f)\leq \dim(\tau)^2$$ for all $f\in L^2(\GaG)$ (cf. \cite[p. 206 (1)]{Knapp1}). Then for all $f_1,f_2\in \scrS^m(\GaG)$,
\begin{align*}
|\langle\rho(a_t)f_1,f_2\rangle| &\leq e^{(\delta-d)t}\sum_{\tau_1,\tau_2\in{\hat K}} \dim(\tau_1)\dim(\tau_2)\|\mathsf{P}_{\tau_1}f_1\|\|\mathsf{P}_{\tau_2}f_2\|
\\&\ll e^{(\delta-d)t} \|f_1\|_{\scrS^m(\GaG)} \|f_2\|_{\scrS^m(\GaG)}
\end{align*}
for $m\in\NN$ large enough, by Lemma \ref{Fourier}.
\end{proof}

\begin{prop}\label{Sobmixing} There exists $m\in\NN$ such that for all $M$-invariant $f_1,f_2\in \scrS^m(\GaG)$,
\begin{equation*}
\lim_{t\rightarrow +\infty} e^{(d-\delta)t}\langle \rho(a_t)f_1,f_2\rangle=m^{\mathrm{BR}}(f_1)\,m^{\mathrm{BR_*}}(f_2).
\end{equation*}
\end{prop}
\begin{proof} Let $m>d(d+1)/4$ satisfy the conclusion of Lemma \ref{mixingbd}. For simplicity, we write $\|f\|_{\scrS^m}=\|f\|_{\scrS^m(\Gamma\ba G)}$.
Using the density of $C^{\infty}_c(\GaG)^M$ in $\scrS^m(\GaG)^M$, given $\epsilon>0$, there exist $f_1^{\epsilon},f_2^{\epsilon}\in C_c^{\infty}(\GaG)$ such that
\begin{equation*}
\|f_i-f_i^{\epsilon}\|_{\scrS^m}\leq \epsilon\qquad i=1,2.
\end{equation*}
We then write, using Lemma \ref{mixingbd},
\begin{align*}
&e^{(d-\delta)t}\langle \rho(a_t)f_1,f_2\rangle\\ &=e^{(d-\delta)t}\Big(\langle \rho(a_t)f_1^{\epsilon},f_2^{\epsilon}\rangle+\langle \rho(a_t)(f_1-f_1^{\epsilon}),f_2\rangle
+\langle \rho(a_t)f_1^{\epsilon},f_2-f^{\epsilon}_2\rangle\Big)
\\ &= e^{(d-\delta)t}\langle \rho(a_t)f_1^{\epsilon},f^{\epsilon}_2\rangle+ O\big(\epsilon(\|f_1\|_{\scrS^m}+\|f_2\|_{\scrS^m})\big).
\end{align*}
By Theorem \ref{roblin}, we have
\begin{equation*}
\lim_{t\rightarrow\infty} e^{(d-\delta)t}\langle \rho(a_t)f_1^{\epsilon},f_2^{\epsilon}\rangle =\mBR(f_1^{\epsilon})\,\mBRast(f_2^{\epsilon}).
\end{equation*}
So
\begin{align*}
\limsup_{t\rightarrow\infty} e^{(d-\delta)t}\langle \rho(a_t)f_1,f_2\rangle=\mBR(f_1^{\epsilon})\,\mBRast(f_2^{\epsilon})+   O\big(\epsilon(\|f_1\|_{\scrS^m}+\|f_2\|_{\scrS^m})\big).
\end{align*}
Since $m> d(d+1)/4$, we may apply  Lemma \ref{mMeasSobbd} to get
\begin{equation*}
\mBR(f_1^{\epsilon})\,\mBRast(f_2^{\epsilon})=\mBR(f_1)\,\mBRast(f_2)+   O\big(\epsilon(\|f_1\|_{\scrS^m}+\|f_2\|_{\scrS^m})\big).
\end{equation*}
Therefore
\begin{align*}
\limsup_{t\rightarrow\infty} e^{(d-\delta)t}\langle \rho(a_t)f_1,f_2\rangle=\mBR(f_1)\,\mBRast(f_2)+   O\big(\epsilon(\|f_1\|_{\scrS^m}+\|f_2\|_{\scrS^m})\big).
\end{align*}
Since $\epsilon>0$ was arbitrary, we in fact have
\begin{align*}
\limsup_{t\rightarrow\infty} e^{(d-\delta)t}\langle \rho(a_t)f_1,f_2\rangle=\mBR(f_1)\,\mBRast(f_2).
\end{align*}
An analogous calculation with $\liminf$ in place of $\limsup$ proves the proposition.
\end{proof}
\section{Proof of Theorem \ref{expmixing}}\label{thmproof}
In this final section, we prove Theorem \ref{expmixing}.
Recall the isomorphism 
$$\big(\rho,L^2(\GaG)\big)\cong \int_{\mathsf{Z}}^{\oplus} (\pi_{\zeta},\scrH_{\zeta})\,d\mu_{\mathsf{Z}}(\zeta).$$ 
Fix arbitrary $0< r< \delta-s_1$. We partition $\mathsf{Z}$  as
\begin{equation*}
\mathsf{Z}=\mathsf{Z}_{r}^-\cup\mathsf{Z}_{r}^+,
\end{equation*}
where
\begin{equation*}
\mathsf{Z}_{r}^-=\left\lbrace \zeta\in\mathsf{Z}\,:\, (\pi_{\zeta},\scrH_{\zeta})\cong \scrU(\upsilon,s),\;\mathrm{where}\;\upsilon\in{\hat M}\;\mathrm{and}\; s\in[s_1+r,\delta]\right\rbrace
\end{equation*}
and $\mathsf{Z}_{r}^+=\mathsf{Z}\setminus\mathsf{Z}_{r}^-$ (cf. \eqref{L2decomp}). This partition is then used to decompose $\big(\rho,L^2(\GaG)\big)$ as
\begin{equation*}
\big(\rho,L^2(\GaG)\big)= \big(\rho,L^2(\GaG)^-\big)\oplus \big(\rho,L^2(\GaG)^+\big),
\end{equation*}
where 
\begin{equation*}
\big(\rho,L^2(\GaG)^{\pm}\big)\cong \int_{\mathsf{Z}_{r}^{\pm}}^{\oplus} (\pi_{\zeta},\scrH_{\zeta})\,d\mu_{\mathsf{Z}}(\zeta)
\end{equation*}
with the dependency on $r$ being slightly suppressed.
Recall that $\scrB_{\delta}$ occurs as a subrepresentation of $\big( \rho, L^2(\GaG)^-\big) $ and $\scrB_{\delta}\cong\scrU(1,\delta)$ (cf. Sections \ref{repdefsec} and \ref{compserdefsec}).  We may further decompose $\big(\rho,L^2(\GaG)^-\big)$ accordingly: let $L^2(\GaG)_0^-$ be the orthogonal complement of $\scrB_{\delta}$ in $L^2(\GaG)^-$. We thus have
\begin{equation*}
\big(\rho,L^2(\GaG)\big)=(\rho,\scrB_{\delta})\oplus\big(\rho,L^2(\GaG)_0^{-}\big)\oplus\big(\rho,L^2(\GaG)^{+}\big).
\end{equation*}
Note that Theorem \ref{spectralgap} and the duality between eigenfunctions of the Laplacian on $\scrM=\GaG/K$ and spherical representations in the decomposition of $L^2(\GaG)$ imply that no spherical representation is weakly contained in $\big(\rho,L^2(\GaG)_0^-\big)$. At the level of functions, we write
\begin{equation*}
f_i=f_i^0+f_i^-+f_i^+,\qquad i=1,2
\end{equation*}
where $ f_i^0\in\scrB_{\delta}$, $f_i^-\in L^2(\GaG)_0^-$, and $f_i^+\in L^2(\GaG)^+$. 

The matrix coefficients we are interested in now decompose as
\begin{equation}\label{matrixcoeffdecomp}
\langle \rho(a_t)f_1,f_2\rangle=\langle \rho(a_t)f_1^0,f_2^0\rangle+\langle \rho(a_t)f_1^-,f_2^-\rangle+\langle \rho(a_t)f_1^+,f_2^+\rangle.
\end{equation}
We deal with the three summands in turn: by construction, $\big(\rho,L^2(\GaG)^+\big)$ does not weakly contain any $\scrU(\upsilon,s)$ with $s>s_1+r$ (this also uses the fact $\big(\rho,L^2(\GaG)\big)$ does not weakly contain any $\scrU(\upsilon,s)$ with $s>\delta$; cf.\ \cite[Proposition 3.23]{MoOh}).

We assume that $m$ is large enough so that all the results of the previous sections hold for $f_1, f_2\in \scrS^m(\GaG)$. By \cite[Proposition 5.3]{KontOh} (cf.\ also \cite[Proposition 3.29]{MoOh} or \cite[Theorem 2.1]{Shalom} combined with the argument from the proof of Lemma \ref{mixingbd}), we have that for all $\xi>0$,
\begin{align}\label{plusbd}
\notag|\langle \rho(a_t)f_1^+,f_2^+\rangle| &\ll_{\xi} e^{(s_1-d +r+\xi)t} \|f_1^+\|_{\scrS^m(\GaG)}\|f_2^+\|_{\scrS^m(\GaG)}
\\&\leq e^{(s_1-d+r+\xi)t} \|f_1\|_{\scrS^m(\GaG)}\|f_2\|_{\scrS^m(\GaG)}.
\end{align}

We will now use Theorem \ref{Minvmcoeffsbdd} to bound $\langle \rho(a_t)f_1^-,f_2^-\rangle$. Let $\widetilde{\mathsf{Z}}_{r}^{-}\subset\mathsf{Z}_{r}^-$ be such that $\big(\rho,L^2(\GaG)_0^{-}\big)\cong \int_{\widetilde{\mathsf{Z}}_{r}^{-}}^{\oplus} (\pi_{\zeta},\scrH_{\zeta})\,d\mu_{\mathsf{Z}}(\zeta)$. The corresponding decomposition of the functions $f_1^-$, $f_2^-$ reads $f_i^-=\int_{\widetilde{\mathsf{Z}}_{\epsilon}^{-}} (f_i^-)_{\zeta}\,d\mu_{\mathsf{Z}}(\zeta)$ ($i=1,2$). Since $f_i$ is $M$-invariant, so is $f_i^-$ and $\mu_{\mathsf{Z}}$-a.\ e.\ $(f_i^-)_{\zeta}$. The matrix coefficient $\langle \rho(a_t)f_1^-,f_2^-\rangle$ is now written as 
\begin{equation*}
\langle \rho(a_t)f_1^-,f_2^-\rangle=\int_{\widetilde{\mathsf{Z}}_{\epsilon}^{-}} \langle\pi_{\zeta}(a_t)(f_1^-)_{\zeta},(f_2^-)_{\zeta}\rangle_{\scrH_{\zeta}}\,d\mu_{\mathsf{Z}}(\zeta).
\end{equation*}

Each $(\pi_{\zeta},\scrH_{\zeta})$ is isomorphic to some $\scrU(\upsilon,s)$ with $\upsilon$ {\it non-trivial} and $s\in[s_1+r,\delta]$.
Setting
\begin{equation*}
\lambda(\delta,r)=\max_{s\in [s_1+r,\delta]} s-d-\eta_s
\end{equation*} where $\eta_s=\min (2s-d,1)$,
we apply Theorem \ref{Minvmcoeffsbdd} to each $\langle\pi_{\zeta}(a_t)(f_1^-)_{\zeta},(f_2^-)_{\zeta}\rangle_{\scrH_{\zeta}}$ and obtain
\begin{align}\label{minusbd}
\notag| &\langle \rho(a_t)f_1^-,f_2^-\rangle|\ll_{r,\delta} e^{\lambda(\delta,r)t}\int_{\widetilde{\mathsf{Z}}_{r}^{-}} \|(f_1^-)_{\zeta}\|_{\scrS^m(\scrH_{\zeta})}\|(f_2^-)_{\zeta}\|_{\scrS^m(\scrH_{\zeta})} \,d\mu_{\mathsf{Z}}(\zeta)
\\\notag \leq& e^{\lambda(\delta,r)t}\sqrt{\int_{\widetilde{\mathsf{Z}}_{r}^{-}} \|(f_1^-)_{\zeta}\|_{\scrS^m(\scrH_{\zeta})}^2\,d\mu_{\mathsf{Z}}(\zeta)}\sqrt{\int_{\widetilde{\mathsf{Z}}_{r}^{-}} \|(f_2^-)_{\zeta}\|_{\scrS^m(\scrH_{\zeta})}^2\,d\mu_{\mathsf{Z}}(\zeta)} 
\\=&  e^{\lambda(\delta,r)t} \|f_1^-\|_{\scrS^m(\GaG)}\|f_2^-\|_{\scrS^m(\GaG)} \notag\\  \leq &e^{\lambda(\delta,r)t} \|f_1\|_{\scrS^m(\GaG)}\|f_2\|_{\scrS^m(\GaG)},
\end{align}
where the implied constant remains uniformly bounded as $r\to 0$.
The remaining term in the left-hand side is $\langle \rho(a_t)f_1^0,f_2^0\rangle$. Using the fact that $(\rho,\scrB_{\delta})$ is isomorphic to $\scrU(1,\delta)$, applying Proposition \ref{trivprop} gives
\begin{align}\label{maintermbd}
\notag\langle \rho(a_t)f_1^0,f_2^0\rangle&= \Phi(f_1^0,f_2^0)\,e^{(\delta-d)t}+O_{\delta}\left( e^{(\delta-d-\eta_{\delta})t}\|f_1^0\|_{\scrS^m(\GaG)}\|f_2^0\|_{\scrS^m(\GaG)}\right)
\\ =& \Phi(f_1^0,f_2^0)\,e^{(\delta-d)t}+O_{\delta}\left( e^{(\delta-d-\eta_{\delta})t}\|f_1\|_{\scrS^m(\GaG)}\|f_2\|_{\scrS^m(\GaG)}\right),
\end{align}
where $\Phi(f_1^0,f_2^0)$ is given by the sum \eqref{maintermsum} under the aforementioned isomorphism.

Note that
\begin{align*}
\beta &:= \min\big\lbrace \eta_{\delta},\delta-s_1-r-\xi,\delta-d-\lambda(\delta,r)\big\rbrace
\\&=\min\left\lbrace 1 ,2\delta-d,\delta-s_1-r-\xi, \min_{s\in[s_1+r,\delta]}\big( \delta-s+\min\lbrace 1, 2s-d\rbrace\big)\right\rbrace
\\&\geq \min\left\lbrace 1 ,2\delta-d,\delta-s_1-r, \min_{s\in[s_1+r,\delta]}\big( \delta-s+\min\lbrace 1, 2s-d\rbrace\big)\right\rbrace-\xi
\\&= \min\left\lbrace 1 ,2\delta-d,\delta-s_1-r,\delta+s_1+r-d\right\rbrace-\xi
\\&= \min\left\lbrace 1,\delta-s_1-r,\delta+s_1+r-d\right\rbrace-\xi\\
&= \min\left\lbrace 1,\delta-s_1-r\right\rbrace-\xi
\\&\geq  \min\left\lbrace 1,\delta-s_1\right\rbrace-(\xi+r).
\end{align*}

Entering \eqref{plusbd}, \eqref{minusbd}, and \eqref{maintermbd} into \eqref{matrixcoeffdecomp} gives
\begin{align*}
e^{(d-\delta)t}\langle \rho(a_t)f_1,f_2\rangle=\Phi(f_1^0,f_2^0)+O_{r,\xi}\big(e^{ - \beta t} \|f_1\|_{\scrS^m(\GaG)}\|f_2\|_{\scrS^m(\GaG)}\big).
\end{align*}

Writing $\eta=\min (1, \delta-s_1)$ and $\epsilon=\xi+r$, we thus have 
\begin{align*}
e^{(d-\delta)t}\langle \rho(a_t)f_1,f_2\rangle=\Phi(f_1^0,f_2^0)+O_{\epsilon}\big(e^{ - (\eta-\epsilon) t} \|f_1\|_{\scrS^m(\GaG)}\|f_2\|_{\scrS^m(\GaG)}\big).
\end{align*}
Choosing $0<\epsilon<\eta$ gives
\begin{equation*}
\lim_{t\rightarrow\infty} e^{(d-\delta)t}\langle \rho(a_t)f_1,f_2\rangle=\Phi(f_1^0,f_2^0),
\end{equation*}
so $\Phi(f_1^0,f_2^0)=m^{\mathrm{BR}}(f_1)\,m^{\mathrm{BR_*}}(f_2)$ by Proposition \ref{Sobmixing}, completing the proof of Theorem \ref{expmixing}.

\begin{Rmk}\label{final}
In the case when $\Gamma$ is a lattice, i.e., when $\delta=d$, the above proof leads to the claim in Remark \ref{mainthmrmk} (1) as follows. Firstly, we have $\langle \rho(a_t)f_1^0,f_2^0\rangle =\int f_1\, dx \cdot \int f_2\, dx $ for any $t\in\RR$. As before, we use \eqref{plusbd} to bound $\langle \rho(a_t)f_1^+,f_2^+\rangle$. Now, since complementary series representations $\scrU(\upsilon,s)$ with \emph{non-trivial} $\upsilon$ exist only for $ s\le d-1$, the constant $\lambda(d,r)$ appearing in the bound \eqref{minusbd} for $|\langle \rho(a_t)f_1^-,f_2^-\rangle |$ is at most $ \alpha:= \max \{s-d-\eta_s : {s\in [s_1+r, d-1]}\}$.
We note that $\alpha<-2+r$, and if, in addition, $\lbrace s>s_1+1: \scrU(\upsilon,s) \textrm{ is weakly contained in } L^2(\GaG)\rbrace=\emptyset$, then one can take $\lambda(d,r) < s_1-d+r$. 
 Hence combining these bounds, we obtain Remark \ref{mainthmrmk} (1).
\end{Rmk}

\end{document}